\documentclass[11pt, leqno]{article}

\usepackage{amsmath,amssymb,amsthm, epsfig}

\usepackage{hyperref}

\usepackage{amsmath}
\usepackage{comment}
\usepackage[pdftex]{color}
\usepackage{stackrel}

\usepackage{mathtools}
\usepackage{color}

\usepackage{ulem}

\usepackage{mathrsfs}

\usepackage{wasysym}

\usepackage{stackrel}

\usepackage{pgfplots}
\pgfplotsset{compat=newest}

\title{\large{\bf Regularity estimates for quasilinear elliptic PDEs in non-divergence form with Hamiltonian terms and applications}}
\author{\it by \smallskip \\ Junior da Silva Bessa \footnote{\noindent Universidade Estadual de Campinas - UNICAMP. Departamento  de Matemática. Campinas - SP, Brazil. \noindent \texttt{E-mail address: \url{jbessa@unicamp.br}}}\quad $\&$\quad Jo\~{a}o Vitor da Silva
\footnote{\noindent Universidade Estadual de Campinas - UNICAMP. Departamento  de Matemática. Campinas - SP, Brazil. \noindent \texttt{E-mail address: \url{jdasilva@unicamp.br}}}
}

\newlength{\hchng}
\newlength{\vchng}
\def \dist {\mathrm{dist}}

\newcommand{\defeq}{\mathrel{\mathop:}=}

\newcommand{\intav}[1]{\mathchoice {\mathop{\vrule width 6pt height 3 pt depth  -2.5pt
\kern -8pt \intop}\nolimits_{\kern -6pt#1}} {\mathop{\vrule width
5pt height 3  pt depth -2.6pt \kern -6pt \intop}\nolimits_{#1}}
{\mathop{\vrule width 5pt height 3 pt depth -2.6pt \kern -6pt
\intop}\nolimits_{#1}} {\mathop{\vrule width 5pt height 3 pt depth
-2.6pt \kern -6pt \intop}\nolimits_{#1}}}

\newtheorem{theorem}{Theorem}[section]

\newtheorem{lemma}[theorem]{Lemma}

\newtheorem{proposition}[theorem]{Proposition}

\newtheorem{corollary}[theorem]{Corollary}

\theoremstyle{definition}

\newtheorem{definition}[theorem]{Definition}

\newtheorem{example}[theorem]{Example}

\theoremstyle{remark}

\newtheorem{remark}[theorem]{Remark}

\numberwithin{equation}{section}

\begin{document}
\maketitle

\begin{abstract}

In this manuscript, we investigate regularity estimates for a class of quasilinear elliptic equations in the non-divergence form that may exhibit degenerate behavior at critical points of their gradient. The prototype equation under consideration is
\[
|\nabla u(x)|^{\theta} \left( \Delta_p^{\mathrm{N}} u(x) + \langle \mathfrak{B}(x), \nabla u \rangle \right) + \varrho(x) |\nabla u(x)|^{\sigma} = f(x) \quad \text{in} \quad B_1,
\]
where $\theta > 0$, $\sigma \in (\theta, \theta + 1)$, and $p \in (1, \infty)$. The coefficients $\mathfrak{B}$ and $\varrho$ are bounded continuous functions, and the source term $f \in \mathrm{C}^0(B_1) \cap L^{\infty}(B_1)$.  We establish interior $\mathrm{C}^{1,\alpha}_{\text{loc}}$ regularity for some $\alpha \in (0,1)$, along with sharp quantitative estimates at critical points of existing solutions. Additionally, we prove a non-degeneracy property and establish both a Strong Maximum Principle and a Hopf-type lemma. In the final part, we apply our analytical framework to study existence, uniqueness, improved regularity, and non-degeneracy estimates for H\'{e}non-type models in the non-divergence form. These models incorporate strong absorption terms and linear/sublinear Hamiltonian terms and are of independent mathematical interest. Our results partially extend (for the non-variational quasilinear setting) the recent work by the second author in collaboration with Nornberg [Calc. Var. Partial Differential Equations 60 (2021), no. 6, Paper No. 202, 40 pp.], where sharp quantitative estimates were established for the fully nonlinear uniformly elliptic setting with Hamiltonian terms. However, our approach deviates from theirs and is founded on a novel perturbative application of the Ishii–Lions technique in order to incorporate the influence of Hamiltonian terms.
	
\vspace{0.2cm}
	
\end{abstract}

\newpage

\section{Introduction}

In this article, we investigate regularity estimates for diffusion problems governed by quasilinear operators in non-divergence form, possibly exhibiting degenerate behavior. Specifically, we focus on prototypical models whose normalized operators arise in Tug-of-War games—stochastic models with noise—where such regularity properties play a fundamental role in understanding the non-variational nature of these problems.

More precisely, the model under consideration is given by
\begin{equation}\label{Problem}
|\nabla u|^{\theta}\Delta_{p}^{\mathrm{N}}u+\mathscr{H}(\nabla u,x)=f(x)\quad\text{in}\quad \mathrm{B}_1.
\end{equation}

Throughout this manuscript, for a fixed \( p \in (1, \infty) \), the operator
\begin{equation}\label{Def-Norm-p-Lap}
\Delta_{p}^{\mathrm{N}}u \defeq |\nabla u|^{2-p} \operatorname{div}(|\nabla u|^{p-2} \nabla u) =\Delta u+(p-2)\left\langle D^2u \frac{\nabla u}{|\nabla u|},\frac{\nabla u}{|\nabla u|}\right\rangle
\end{equation}
is referred to as the \textit{normalized \( p \)-Laplacian} or the \textit{game-theoretic \( p \)-Laplace operator}, which arises in certain stochastic models (see \cite{BlancRossi-Book}, \cite{Lewicka20}, and \cite{Parvia2024} for comprehensive surveys on Tug-of-War games and their associated PDEs). Moreover, the fully nonlinear operator
\[
\Delta_{\infty}^{\mathrm{N}} u(x) \coloneqq \left\langle D^2 u \frac{\nabla u}{|\nabla u|}, \frac{\nabla u}{|\nabla u|} \right\rangle
\]
is known as the \textit{normalized infinity-Laplacian}, which also plays a significant role in the interplay between nonlinear PDEs and Tug-of-War games (cf. \cite{BlancRossi-Book} and \cite{Parvia2024}).

From this point onward, we adopt the following assumptions:

\begin{itemize}
\item[(\textbf{H1})] \( p \in (1, \infty) \) is a fixed parameter related to the elliptic nature of the operator;
\item[(\textbf{H2})] \( \theta \in (0, \infty) \) characterizes the degeneracy of the model;
\item[(\textbf{H3})] The (linear-sublinear) Hamiltonian term \(\mathscr{H}:\mathbb{R}^{n}\times \mathrm{B}_{1}\to \mathbb{R}\) is given by
\[
\mathscr{H}(\xi,x)=\langle\mathfrak{B}(x),\xi\rangle|\xi|^{\theta}+\varrho(x)|\xi|^{\sigma},
\]
 where \( \sigma \in (\theta,   1+\theta)\), \(\mathfrak{B}\in \mathrm{C}^{0}(\overline{\mathrm{B}_{1}};\mathbb{R}^{n})\), and \(\varrho\in \mathrm{C}^{0}(\overline{\mathrm{B}_{1}})\).
 \item[(\textbf{H4})] The source term \( f: \mathrm{B}_1 \to \mathbb{R} \) satisfies \( f \in L^{\infty}(\mathrm{B}_1) \cap \mathrm{C}^{0}(\mathrm{B}_1) \).
\end{itemize}

Additionally, we impose the boundary condition
\begin{equation}\label{Bound-Cond}
 u(x) = g(x) \quad \text{on } \partial \Omega,
\end{equation}
where the boundary datum is assumed to be continuous, i.e., \( g \in \mathrm{C}^0(\partial \Omega) \).

This work extends the theory of regularity for quasilinear elliptic equations in non-divergence form by addressing scenarios where the first-order term exhibits linear or sublinear growth—a setting that remains less understood and more challenging compared to the inhomogeneous problem without first-order contributions (cf. \cite{Attou18}). Our results have implications for the qualitative behavior of solutions to such PDEs and contribute to the broader understanding of nonlinear elliptic models and related geometric free boundary problems (see, for instance, Theorems \ref{exist_uniq}, \ref{Higher_continuity}, and \ref{NDHTE}).

It is important to emphasize the absence of a variational structure in equation \eqref{Def-Norm-p-Lap}, which represents a uniformly elliptic operator in trace form. This contrasts with the classical \( p \)-Laplacian, a quasilinear operator in divergence form.

We assert that the normalized \( p \)-Laplacian is “uniformly elliptic” in the sense that
\[
\mathcal{M}^{-}_{\lambda, \Lambda}(D^2 u) \leq \Delta_p^{\mathrm{N}} u \leq \mathcal{M}^{+}_{\lambda, \Lambda}(D^2 u),
\]
where
\begin{equation}\label{Pucci}\tag{Pucci}
\mathcal{M}^{-}_{\lambda, \Lambda}(D^2 u) \coloneqq \inf_{\mathfrak{A} \in \mathcal{A}_{\lambda, \Lambda}} \text{tr}(\mathfrak{A} D^2 u) \quad \text{and} \quad 
\mathcal{M}^{+}_{\lambda, \Lambda}(D^2 u) \coloneqq \sup_{\mathfrak{A} \in \mathcal{A}_{\lambda, \Lambda}} \text{tr}(\mathfrak{A} D^2 u)
\end{equation}
are the \textit{Pucci extremal operators}, and for \( 0< \lambda\leq \Lambda< \infty \),
\[
\mathcal{A}_{\lambda, \Lambda} \defeq \left\{\mathfrak{A} \in \text{Sym}(n): \,\, \lambda|\xi|^2 \leq \left\langle  \mathfrak{A} \xi, \xi \right\rangle \leq \Lambda|\xi|^2 \quad \text{for all } \xi \in \mathbb{R}^n \right\},
\]
where \( \text{Sym}(n) \) denotes the set of symmetric \( n \times n \) matrices.

It is well known that the Normalized \( p \)-Laplacian operator can be written in the form
\[
\Delta_p^{\mathrm{N}} u = \mathrm{tr}\left[\left(\mathrm{Id}_n + (p - 2) \frac{\nabla u}{|\nabla u|} \otimes \frac{\nabla u}{|\nabla u|}\right) D^2 u\right],
\]
from which one can readily verify that 
\[
\lambda = \lambda^{\mathrm{N}}_p = \min\left\{1, p - 1\right\} \quad \text{and} \quad \Lambda = \Lambda^{\mathrm{N}}_p = \max\left\{1, p - 1\right\}
\]
are its smallest and largest eigenvalues.

In summary, this work establishes the following results concerning \eqref{Problem}:

\begin{enumerate}
\item[\checkmark] Existence of viscosity solutions (via a Comparison Principle -Lemma \ref{Comp-Princ});
\item[\checkmark] Compactness of viscosity solutions (via an adjusted Ishii–Lions method - Lemma \ref{Holderest});
\item[\checkmark] Approximation schemes for translated solutions (deviation by planes) - Lemma \ref{approxlemma}; 
\item[\checkmark] Gradient estimates for viscosity solutions - Theorem \ref{Thm1.1};
\item[\checkmark] Optimal estimates along specific sets and for planar solutions - Corollaries \ref{Corol01} and \ref{Corol02};
\item[\checkmark] Optimal estimates at critical points of existing solutions - Theorem \ref{Opt_reg-Extremum-point};  
\item[\checkmark] Non-degeneracy estimates for solutions - Theorem \ref{Thm-Non-Deg};
\item[\checkmark] A strong maximum principle and a Hopf-type lemma for equations with borderline-homogeneous first-order terms - Theorem \ref{thm:SMP};
\item[\checkmark] Existence and uniqueness of solutions to Hénon-type problems with strong absorption - Theorem \ref{exist_uniq}; 
\item[\checkmark] Optimal estimates for related Hénon-type problems with strong absorption - Theorem \ref{Higher_continuity};
\item[\checkmark] Non-degeneracy of solutions to Hénon-type problems - Theorem \ref{NDHTE}.
\end{enumerate}

To the best of our knowledge, these aspects have not yet been thoroughly addressed in the existing literature and play a pivotal role in the regularity theory of nonlinear models. Consequently, our results contribute to the development of the theory of quasilinear partial differential equations in non-divergence form and their associated topics, including geometric free boundary problems, blow-up analysis, and classification of global profiles (Liouville-type results).

\medskip

Before stating our main results, we introduce some notations and function classes that will be useful throughout our analysis.

For subsequent developments, for \( p \in (1, \infty) \), we define the class:
\[
\mathbb{H}_{n}^{(p)} \defeq
\left\{
\mathfrak{h} \in W^{1, p}(\mathrm{B}_1) \cap \mathrm{C}^0(\mathrm{B}_1) \,\, \middle| \,\, -\Delta_p \mathfrak{h} = 0 \quad \text{in} \quad \mathrm{B}_1 \subset \mathbb{R}^n \text{ (in the weak sense)}
\right\}.
\]

We note that, although this constitutes a broad class of functions, it is well established (see Section \ref{Section02} for a detailed discussion) that there exists a universal modulus of continuity for the gradient of functions in \( \mathbb{H}_{n}^{(p)} \) (see, for instance, \cite{DiBe}, \cite{DuzMing10}, \cite{DuzMing11}, \cite{Lewis83}, \cite{Tolks}, and \cite{Ural68} for an incomplete list of contributions). Specifically, if \( \mathfrak{h} \in \mathbb{H}_{n}^{(p)} \), then there exist constants \( \mathrm{C}^{\star} > 0 \) and \( \alpha^{\star} \in (0, 1) \), depending only on the dimension and on \( p \), such that
\begin{equation}\label{EqHomogeneous}
\|\mathfrak{h}\|_{\mathrm{C}^{1,\alpha^{\star}}(\mathrm{B}_{1/2})} \leq \mathrm{C}^{\star} \cdot \|\mathfrak{h}\|_{L^{\infty}(\mathrm{B}_1)}.
\end{equation}

Additionally, it is important to highlight that the exponent \( \alpha^{\star} \in (0, 1) \) represents a theoretical bound in the regularity theory of weak solutions to
\[
- \Delta_p \mathfrak{h} = 0 \quad \text{in} \quad \mathrm{B}_1,
\]
which can be characterized as follows:
\begin{equation}\label{Def_alpha_Hom}
\alpha_{\mathrm{Hom}}^{(p)} \defeq
\sup \left\{
\begin{array}{ccc}
\alpha^{\star} \in \left(0, 1\right) & | & \text{there exists} \,\,\, \mathrm{C}_{n, p, \alpha^{\star}} > 0 \, \text{such that} \, \eqref{EqHomogeneous} \\
& &  \text{holds with} \,\,\, \mathrm{C}_{n, p, \alpha^{\star}} = \mathrm{C}^{\star} \, \text{for all} \, \mathfrak{h} \in \mathbb{H}_{n}^{(p)} 
\end{array}
\right\}.
\end{equation}

The following result, established by Siltakoski, provides a fundamental connection between viscosity solutions of the (homogeneous) normalized \( p \)-Laplacian and weak solutions of the (homogeneous) \( p \)-Laplacian.

\begin{theorem}[{\bf \cite[Theorem 5.9]{Siltak18}}]\label{Thm-Equiv-Sol}
A function \( u \in \mathrm{C}^0(\Omega)\) is a viscosity solution to  
\[
- \Delta^{\mathrm{N}}_{p} u = 0 \quad \text{in } \quad \Omega
\]  
if and only if it is a weak solution to  
\[
- \Delta_{p} u = 0 \quad \text{in } \quad \Omega.
\]
\end{theorem}

This result serves as a key tool to relate the regularity estimates for viscosity solutions of \eqref{Problem} with known regularity results for the associated homogeneous geometric tangential equation driven by the \( p \)-Laplacian (via an approximation scheme), which enjoys a well-developed elliptic regularity theory.  

\subsection*{Statement of the main results and some consequences}

In our first result, we will establish the existence of viscosity solutions to \eqref{Problem}.

\begin{theorem}[\bf Existence of Solutions]\label{Exist_Uniq_sol}
Suppose the assumptions $(\mathbf{H1})-(\mathbf{H4})$ are satisfied. Additionally, assume that either $\displaystyle \inf_{\Omega} f(x) > 0$ or $\displaystyle \sup_{\Omega} f(x) < 0$. Suppose that there exists a viscosity subsolution $u_{\flat} \in \mathrm{C}^0(\Omega) \cap \mathrm{C}_{\text{loc}}^{0,1}(\Omega)$ and a viscosity supersolution $u^{\sharp} \in \mathrm{C}^0(\Omega) \cap \mathrm{C}_{\text{loc}}^{0,1}(\Omega)$ to the equation \eqref{Problem} such that $u_{\flat} = u^{\sharp} = g \in \mathrm{C}^0(\partial \Omega)$. Define the class of functions
\begin{equation}
\mathcal{S}_g(\Omega) := \left\{ v \in \mathrm{C}^0(\Omega) \mid v \text{ is a viscosity supersolution to } \eqref{Problem}, \right. 
\end{equation}
\begin{equation*}
    \left. \text{such that } u_{\flat} \leq v \leq u^{\sharp} \text{ and } v = g \text{ on } \partial \Omega \right\}.
\end{equation*}
Then, the function
\begin{equation}
    u(x) := \inf_{\mathcal{S}_g(\Omega)} v(x), \quad \text{for } x \in \Omega,
\end{equation}
is a continuous (up to the boundary) viscosity solution to the problem
\begin{equation}
\left\{
    \begin{array}{rclcl}
        |\nabla u(x)|^{\theta} \left( \Delta_p^{\mathrm{N}} u(x) + \langle \mathfrak{B}(x), \nabla u \rangle \right) + \varrho(x) |\nabla u(x)|^{\sigma} & = & f(x) & \text{in } & \Omega, \\
        u(x) & = & g(x) & \text{on }& \partial \Omega.
    \end{array}
\right.
\end{equation}
\end{theorem}

\begin{example}[{\bf Non-Uniqueness of solutions}]
In general, the uniqueness of solutions to the Dirichlet problem associated with equation \eqref{Problem} may fail when the source term does not have a definite sign. For example, consider \(u(x) = 0\) and \(v(x) = \mathrm{c}(1 - |x|^{\beta})\) in \(\mathrm{B}_1\), where \(\beta = \frac{2 + \theta}{1 + \theta}\) and \(\mathrm{c} = \beta^{-1}\). Then, observe that both \(u\) and \(v\) are viscosity solutions to
$$
\left\{
\begin{array}{rclcl}
|\nabla w|^{\theta} \Delta_p^{\mathrm{N}} w+\mathscr{H}(\nabla w,x) & = & 0 & \text{in } &B_{1}, \\
w(x) & = & 0 & \text{on }  & \partial B_{1},
\end{array}
\right.
$$
where \(\mathfrak{B}(x) = \frac{p-2}{1+\theta}|x|x\), \(\varrho(x) = \frac{n+(n-1)\theta}{1+\theta}|x|^{1+\theta-\sigma}\), and \(\theta < \sigma < 1 + \theta\).
\end{example}

In our next result, we establish a gradient regularity estimate for solutions to \eqref{Problem}, namely, $\mathrm{C}_{\text{loc}}^{1, \alpha}$-estimates (for some universal $\alpha \in (0, 1)$).

\begin{theorem}[{\bf Gradient Regularity Estimate}]\label{Thm1.1}
Suppose that the conditions $(\mathbf{H1})-(\mathbf{H4})$ are satisfied and let $u \in \mathrm{C}^0(\mathrm{B}_1)$ be a bounded viscosity solution to problem \eqref{Problem}. Given an
\[
\alpha \in \left( 0, \alpha_{\mathrm{Hom}}^{(p)} \right) \cap \left( 0, \frac{1}{1 + \theta} \right]
\]
we have that \(u \in \mathrm{C}_{\text{loc}}^{1, \alpha}(\mathrm{B}_1)\). Moreover, the following regularity estimate holds:
$$
[u]_{\mathrm{C}^{1, \alpha}(\mathrm{B}_{1/2})} \leq \mathrm{C} \left( \|u\|_{L^{\infty}(\mathrm{B}_{1})} + \|\varrho\|^{\frac{1}{1+\theta-\sigma}}_{L^{\infty}(\mathrm{B}_{1})} + \|f\|^{\frac{1}{1+\theta}}_{L^{\infty}(\mathrm{B}_{1})} \right)
$$ 
where \(\mathrm{C} > 0\) depends only on \(n\), \(p\), \(\sigma\), \(\theta\), and \(\|\mathfrak{B}\|_{L^{\infty}(\mathrm{B}_1; \mathbb{R}^{n})}\).
\end{theorem}

\bigskip

It is worth emphasizing that our method for establishing gradient regularity estimates emerges naturally from a refined argument involving controlled translations (inspired by \cite{daSNorb21}), where we approximate solutions by affine functions. The proof strategy involves four fundamental steps:

\begin{itemize}
    \item[\textbf{Step 1.}] Compactness properties of translated solutions under appropriate control conditions on the coefficients of the Hamiltonian $\mathscr{H}$:
    $$
    |\nabla w + \vec{q}|^{\theta} \Delta_{p}^{\mathrm{N}}(w + \langle \vec{q}, x \rangle) + \mathscr{H}(\nabla w + \vec{q}, x) = \bar{f}(x) \quad \text{in} \quad \mathrm{B}_1, \quad \text{with} \quad \vec{q} \in \mathbb{R}^n.
    $$   
    Moreover, for some universal $\eta_0 > 0$ and $\mu \in (0, 1)$, the following holds via the Ishii-Lions Lemma:
    $$
    \max \left\{ \|\mathfrak{B}\|_{L^{\infty}(\mathrm{B}_1; \mathbb{R}^n)}(1 + |\vec{q}|), \|\varrho\|_{L^{\infty}(\mathrm{B}_1)}(1 + |\vec{q}|^{\sigma - \theta}) \right\} \leq \eta_0 \quad \Rightarrow \quad [w]_{\mathrm{C}^{0, \mu}(\mathrm{B}_{\frac{15}{16}})} \leq \mathrm{C} < \infty.
    $$

    \item[\textbf{Step 2.}] Under suitable smallness assumptions on both the source term $f$ and the components of the Hamiltonian, i.e., $\mathfrak{B}$ and $\varrho$, solutions to the translated problem admit approximation by $p$-harmonic or harmonic profiles:
    $$
    |\nabla w + \vec{q}|^{\theta} \Delta_{p}^{\mathrm{N}}(w + \langle \vec{q}, x \rangle) + \mathscr{H}(\nabla w + \vec{q}, x) = \bar{f}(x) \quad \text{in} \quad \mathrm{B}_1, \quad \text{with} \quad \vec{q} \in \mathbb{R}^n.
    $$ 
    Furthermore, for some universal $\eta > 0$, we have the following:
    $$
    \max \left\{ \|\bar{f}\|_{L^{\infty}(\mathrm{B}_1)}, \|\mathfrak{B}\|_{L^{\infty}(\mathrm{B}_1; \mathbb{R}^n)}(1 + |\vec{q}|), \|\varrho\|_{L^{\infty}(\mathrm{B}_1)}(1 + |\vec{q}|^{\sigma - \theta}) \right\} \leq \eta \quad \Rightarrow \quad
    \left\{
    \begin{array}{l}
    \displaystyle \sup_{\mathrm{B}_{1/2}} |w - \mathfrak{h}| \ll 1, \\
    \text{where} \\
    \Delta_p \mathfrak{h} = 0 \\
    \text{or} \\
    \mathrm{tr}(\mathfrak{A} D^2 \mathfrak{h}) = 0.
    \end{array}
    \right.
    $$

    \item[\textbf{Step 3.}] By performing a geometric iteration, we establish pointwise $\mathrm{C}^{1,\alpha}$ regularity for solutions to the problem \eqref{Problem}: there exists a sequence of affine functions $\{\mathfrak{l}_k\}_{k \in \mathbb{N}}$ such that
    $$
    \max \left\{ \|f\|_{L^{\infty}(\mathrm{B}_1)}, \|\mathfrak{B}\|_{L^{\infty}(\mathrm{B}_1; \mathbb{R}^n)}, \|\varrho\|_{L^{\infty}(\mathrm{B}_1)} \right\} \leq \delta_0(\text{universal}).
    $$ 
    Then, we have
    $$
 \displaystyle\sup_{\mathrm{B}_{\rho^k}(x_0)} \frac{\left|u(x)-\mathfrak{l}_k(x)\right|}{\rho^{k(1+\alpha)}}\leq 1 \quad \stackrel[\Longrightarrow]{\text{Dini-Campanato }}{{\text{embedding}}} \,\,\,u \,\, \text{is} \,\,\,\mathrm{C}^{1, \alpha} \quad \text{at}\,\,\,x_0.
$$

    \item[\textbf{Step 4.}] A normalization, scaling process, and a standard covering argument reduce the general setting to the analysis of the previous steps.
\end{itemize}


As a result of Theorem \ref{Thm1.1}, we obtain sharp regularity estimates in some specific scenarios.

\begin{corollary}\label{Corol01}
Let \( \mathrm{B}_1 \subset \mathbb{R}^n \), and let \( u \in \mathrm{C}^0(\mathrm{B}_1) \) be a bounded viscosity solution of \eqref{Problem}. Suppose further that \( \alpha^{(p)}_{\mathrm{Hom}} \in \left(\frac{1}{1+\theta}, 1\right] \). Then, \( u \in \mathrm{C}_{\text{loc}}^{1, \frac{1}{1+\theta}}(\mathrm{B}_1) \) and
\[
[u]_{\mathrm{C}^{1, \frac{1}{\theta+1}}(\mathrm{B}_{1/2})} \leq \mathrm{C}(\text{universal}) \left( \|u\|_{L^{\infty}(\mathrm{B}_1)} + \|\varrho\|_{L^{\infty}(\mathrm{B}_1)}^{\frac{1}{1+\theta-\sigma}} + \|f\|_{L^\infty(\mathrm{B}_1)}^{\frac{1}{\theta+1}} \right).
\]
\end{corollary}

Moreover, we may present a sharp estimate for viscosity solutions in the plane (cf. \cite{ATU17} and \cite{LL}).

\begin{corollary}[{\bf Optimal estimates in $2$-$\mathrm{D}$}]\label{Corol02}
Let \( \mathrm{B}_1 \subset \mathbb{R}^2 \), and let \( u \in \mathrm{C}^0(\mathrm{B}_1) \) be a bounded viscosity solution of \eqref{Problem}. Suppose further that $\theta > \left(\frac{1}{p-1} + \tau_0\right)^{-1} - 1$ for $p > 2$, and a constant \( \tau_0 < \frac{p-2}{p-1} \) (see Proposition \ref{PropB-K}). Then, \( u \in \mathrm{C}_{\text{loc}}^{1, \frac{1}{\theta+1}}(\mathrm{B}_{1/2}) \) and
\[
[u]_{\mathrm{C}^{1, \frac{1}{\theta+1}}(\mathrm{B}_{1/2})} \leq \mathrm{C}_p \left( \|u\|_{L^{\infty}(\mathrm{B}_1)} + \|\varrho\|_{L^{\infty}(\mathrm{B}_1)}^{\frac{1}{1+\theta-\sigma}} + \|f\|_{L^\infty(\mathrm{B}_1)}^{\frac{1}{1+\theta}} \right).
\]
\end{corollary}

It is worth emphasizing that our results remain notable even in the case $p = 2$. In this setting, we consider the prototypical model
\begin{equation}\label{LinearModel}
\mathscr{L}[u] \coloneqq |\nabla u(x)|^{\theta} \left( \Delta u(x) + \langle \mathfrak{B}(x), \nabla u \rangle \right) + \varrho(x) |\nabla u(x)|^{\sigma} = f(x) \quad \text{in} \quad \mathrm{B}_1.
\end{equation}
In this framework, due to Lemma \ref{CT-Lemma} and Theorem \ref{Thm-Equiv-Sol}, the associated homogeneous problem with vanishing Hamiltonian—also known as the geometric tangential equation—reduces to the Laplace equation, i.e., $-\Delta \mathfrak{h} = 0$. Consequently, the function $\mathfrak{h}$ admits local $\mathrm{C}^{1,1}$ regularity estimates. Therefore, by Corollary \ref{Corol01}, solutions to \eqref{LinearModel} belong to the class $\mathrm{C}_{\text{loc}}^{1, \frac{1}{\theta+1}}(\mathrm{B}_1)$. 

Additionally, our results are novel even in the case of quasilinear models without the degeneracy law $\xi \mapsto \Phi(|\xi|) = |\xi|^{\theta} \equiv 1$, as well as for equations involving linear or sublinear Hamiltonian terms. Specifically, when we consider the model equation
\[
\mathcal{L}_p[u] \coloneqq  \Delta_p^{\mathrm{N}} u(x) + \langle \mathfrak{B}(x), \nabla u \rangle  + \varrho(x) |\nabla u(x)|^{\sigma} = f(x) \quad \text{in} \quad \mathrm{B}_1,
\]
Theorem~\ref{Thm1.1} ensures that viscosity solutions exhibit regularity properties closely aligned with those of $p$-harmonic functions, namely, $u \in \mathrm{C}_{\text{loc}}^{1, \alpha}$ for every $\alpha \in (0, \alpha^{(p)}_{\mathrm{Hom}})$. This represents a significant contribution to the theoretical framework surrounding quasilinear equations in non-divergence form (cf. \cite{APR17} for the corresponding results on the normalized $p$-Poisson problem).

These regularity estimates can be viewed as the Hamiltonian/degenerate counterparts of the results established by Araújo \textit{et al.} in \cite{ART15}, Imbert-Silvestre in \cite{IS13}, and da Silva–Nornberg in \cite{daSNorb21}, within the fully nonlinear context. Furthermore, our findings are closely related to the works of Birindelli and Demengel—specifically \cite{BirDem14}, which addresses the linear regime, and \cite{BirDem16}, which treats the sublinear case. Finally, it is also worthwhile to highlight the connection with the work of Attouchi and Ruosteenoja \cite{Attou18}, which establishes gradient estimates for degenerate quasilinear elliptic models in non-divergence form, thereby underscoring the significant advances introduced in our study.

We now derive a sharp estimate at the critical points of existing solutions.

\begin{theorem}[{\bf Sharp regularity estimates}]\label{Opt_reg-Extremum-point}
Suppose the assumptions of Theorem \ref{Thm1.1} are satisfied. Assume that \( x_0 \in \mathrm{B}_1 \) is a local extremum point of \( u \). Then,
\begin{equation}\label{Higher Reg}
\sup_{x \in \mathrm{B}_r(x_0)} |u(x)-u(x_0)| \leq \mathrm{C} r^{1 + \frac{1}{1+\theta}},
\end{equation}
for \( r \in (0, 1/2) \), where \( \mathrm{C} > 0 \) is a universal constant.
\end{theorem}

\medskip

To clarify, sharp regularity estimates for solutions to certain elliptic models are fundamental tools in the qualitative analysis of partial differential equations. Indeed, such estimates play a decisive role in various aspects, including:
\begin{itemize}
    \item[(i)] Controlling the growth rate of solutions away from their free boundaries;
    \item[(ii)] Determining the precise vanishing rate near singular sets of solutions;
    \item[(iii)] Establishing the convergence of blow-up profiles;
    \item[(iv)] Classifying global solutions (Liouville-type results);
    \item[(v)] Obtaining uniform estimates to prove existence results via topological methods. 
\end{itemize}

\bigskip

From a geometric viewpoint, it is a piece of pivotal qualitative information to obtain the (counterpart) sharp lower bound estimate as the one addressed in Theorem \ref{Opt_reg-Extremum-point}.

Therefore, under a natural, non-degeneracy assumption on source term, we further obtain the precise behavior of solutions at certain interior critical sets.

\begin{theorem}[{\bf Non-degeneracy Estimate}]\label{Thm-Non-Deg}
Let \( u \in \mathrm{C}^0(\mathrm{B}_1) \) be a bounded viscosity solution to \eqref{Problem}. Assume that \( \displaystyle \inf_{\mathrm{B}_{1}} f \geq \mathfrak{c}_{0} \) for some constant \( \mathfrak{c}_{0} > 0 \). Then, for any local extremum point \( x_0 \in \mathrm{B}_{1} \) and for every \( r > 0 \) such that \( \mathrm{B}_r(x_0) \subset \mathrm{B}_{1} \), we have
\[
\sup_{\partial \mathrm{B}_r(x_{0})}(u - u(x_{0})) \geq \mathrm{C}(n, p, \theta, \sigma, \|\mathfrak{B}\|_{L^{\infty}(\mathrm{B}_{1};\mathbb{R}^{n})}, \|\varrho\|_{L^{\infty}(\mathrm{B}_{1})}, \mathfrak{c}_{0}) \cdot r^{1+\frac{1}{1+\theta}}.
\]
\end{theorem}

We next establish a Strong Maximum Principle and a Hopf-type lemma for an elliptic model involving linear first-order terms.

\begin{theorem}[{\bf Strong Maximum Principle}]\label{thm:SMP}
Let \( \Omega \subset \mathbb{R}^n \) be a bounded domain, and let \( \mathfrak{B} \) and \( \mathfrak{c} \leq 0 \) be continuous functions on \( \overline{\Omega} \). Suppose that \( v \in \mathrm{C}_{\mathrm{loc}}^{0,1}(\Omega) \) is a non-negative viscosity supersolution of
\begin{equation}
|\nabla v|^{\theta}\Delta^{\mathrm{N}}_p v + \mathscr{H}(\nabla v, x) + \mathfrak{c}(x)v^{1+\theta}(x) = 0, \quad x \in \Omega,
\end{equation}
where \( \sigma \in (\theta, \infty) \). In the sublinear case \( \theta < \sigma < 1+\theta \), assume in addition that \( \varrho \geq 0 \). Then either \( v \equiv 0 \) or \( v > 0 \) in \( \Omega \). Furthermore, if \( \Omega \) satisfies the interior sphere condition and \( v(x) > v(z) = 0 \) for all \( x \in \Omega \) and some \( z \in \partial \Omega \), then there exists a constant \( \nu > 0 \) such that
\begin{equation}
v(x) \geq \nu(r - |x - x_0|), \quad \text{for } x \in B_r(x_0),
\tag{2.7}
\end{equation}
where \( B_r(x_0) \subset \Omega \) is a ball tangent to \( z \).
\end{theorem}




\medskip

\subsection*{Applications to H\'{e}non-type models with strong absorption}

A natural question arising in the study of free boundary problems associated with \eqref{Problem} concerns the determination of the optimal regularity of viscosity solutions along certain level sets. For example, by fixing \( \mathrm{R} > r > 0 \) and defining
\begin{equation}\label{beta-exponent}
\hat{\beta} = \hat{\beta}(\theta, m) := \frac{2+\theta}{1+\theta - m},
\end{equation}
one may consider the radial profile \( \Phi_{r,\hat{\beta}} : \mathrm{B}_{\mathrm{R}}(x_0) \to \mathbb{R}^+ \) given by
\begin{equation}\label{Radial-profile}
    \Phi_{r,\hat{\beta}}(x) = \mathrm{c}_{n,\theta,\hat{\beta},p} \left(|x - x_0| - r\right)^{\hat{\beta}}_+,
\end{equation}
where
\[
\mathrm{c}_{n,\theta,\hat{\beta},p} = \mathrm{c}(n,\theta,\hat{\beta},p) = \frac{1}{\hat{\beta}} \cdot \frac{1}{\left[(p - 1)(\hat{\beta} - 1)+1\right]^{\frac{1}{1+\theta}}} > 0
\]
is an appropriate constant, and
\begin{eqnarray*}
\mathfrak{B}(x) =
\begin{cases}
-(n-1)\dfrac{x - x_{0}}{|x - x_{0}|^{2}}, & \text{if } x \in \mathrm{B}_{R}(x_{0}) \setminus \mathrm{B}_{r}(x_{0}),\\[1.5ex]
-(n-1)\dfrac{x - x_{0}}{r^{2}}, & \text{if } x \in \overline{\mathrm{B}_{r}(x_{0})}
\end{cases}
\end{eqnarray*}
and
\[
\varrho(x) = \left(\mathrm{c}_{n,\theta,\hat{\beta},p} \hat{\beta}\right)^{1+\theta - \sigma} \left(|x - x_0| - r\right)_{+}^{(\hat{\beta} - 1)(1+\theta - \sigma) - 1}.
\]
Moreover, we must assume the condition \( \sigma \leq \frac{m(2+\theta)}{m+1} \), which, under assumption \((\textbf{H3})\), implies that \( m > \frac{\theta}{2} \), establishing a new lower bound.

Thus, it is straightforward to verify that \( \Phi_{r,\hat{\beta}} \) is a viscosity solution to the equation
\[
|\nabla \Phi_{r,\hat{\beta}}(x)|^{\theta} \Delta_p^{\mathrm{N}} \Phi_{r,\hat{\beta}}(x) + \mathscr{H}(\nabla \Phi_{r,\hat{\beta}}(x),x) = \left(\Phi_{r,\hat{\beta}}(x)\right)_{+}^{m} \quad \text{in} \quad \mathrm{B}_{\mathrm{R}}(x_0).
\]

On the other hand, from Theorem \ref{Thm1.1}, it is well known that, in general, solutions to \eqref{Problem} belong to \( \mathrm{C}^{1,\alpha}_{\text{loc}} \) for some \( \alpha \in (0,1) \). However, for this particular example, with \( \theta > 0 \) fixed, one observes that \( \Phi_{r,\hat{\beta}} \in \mathrm{C}^{\hat{\beta}}_{\text{loc}} \) along free boundary points, where \( \hat{\beta} > 1 \) is given by \eqref{beta-exponent}.

Notably, when \( m > 0 \), we obtain
\[
\hat{\beta}(\theta, m) = \frac{2+\theta}{1+\theta - m} > 1 + \frac{1}{1+\theta} \quad \Rightarrow \quad \text{enhanced smoothness along the free boundary}.
\]
In other words, within a small radius, viscosity solutions satisfy the sharp growth estimates:
\begin{equation}\label{Growth-Control}
\varrho^{\hat{\beta}} \lesssim \sup_{\mathrm{B}_{\varrho}(x_0)} \omega(x) \lesssim \varrho^{\hat{\beta}}.
\end{equation}

In this final part, we investigate Hénon-type equations governed by quasilinear elliptic PDEs in non-divergence form, incorporating linear or sublinear Hamiltonians and strong absorption terms, as follows:
\begin{equation}\label{Eq-Henon-type}
|\nabla u|^\theta \Delta_p^\mathrm{N} u(x) + \mathscr{H}(\nabla u,x) = \mathfrak{f}(x) u_+^{m}(x) \quad \text{in} \quad \mathrm{B}_1,
\end{equation}
where the weight function \( \mathfrak{f} : \mathrm{B}_1 \to \mathbb{R}_+ \) satisfies the following condition: given \( \alpha > 0 \) and a closed subset \( \mathrm{F} \Subset \mathrm{B}_1 \), we assume
\begin{equation}\label{Homo_cond}
\mathrm{dist}(x, \mathrm{F})^{\alpha} \lesssim \mathfrak{f}(x) \lesssim \mathrm{dist}(x, \mathrm{F})^{\alpha}.
\end{equation}

\medskip

The mathematical investigation of equation \eqref{Eq-Henon-type} is of considerable importance, not only due to its potential applications but also because of its intrinsic connection to a variety of free boundary problems that have been extensively studied in the literature (see, for instance, Alt–Phillips \cite{Alt-Phillips86}, D\'{i}az \cite{Diaz85}, Friedman–Phillips \cite{Fried-Phill84}, Phillips \cite{Phillips83}, and Shahgholian et al. \cite{KKPS2000, Lee-Shahg2003} for related variational problems. For non-variational counterparts, refer to Da Silva et al. \cite{ADaSS25, daSLR21} and Teixeira \cite{Teix16}).

The results presented in this section constitute a natural extension of those addressed by Alcantara \textit{et al.} in \cite{ADaSS25}, concerning problems governed by quasilinear elliptic models in non-divergence form with strong absorption terms, such as \eqref{Eq-Henon-type}, but without Hamiltonian contributions.

\medskip

We now turn our attention to the existence of viscosity solutions for the Dirichlet problem \eqref{Eq-Henon-type}. This existing result is obtained via Perron's method, under the assumption that a suitable version of the Comparison Principle holds. Specifically, consider the functions \(\overline{u}\) and \(\underline{u}\), which solve the following boundary value problems:
\begin{equation}
\left\{
\begin{array}{rclcl}
|\nabla \overline{u}|^{\theta} \Delta_p^{\mathrm{N}} \overline{u}+\mathscr{H}(\nabla \overline{u},x) & = & 0 & \text{in } &\Omega, \\
\overline{u}(x) & = & g(x) & \text{on }  & \partial \Omega,
\end{array}
\right.
\end{equation}
and
\begin{equation}
 \left\{
\begin{array}{rclcl}
|\nabla \underline{u}|^{\theta} \Delta_p^{\mathrm{N}} \underline{u}+\mathscr{H}(\nabla\underline{u},x) & = & \|\mathfrak{f}\|_{L^{\infty}(\Omega)}\|g\|_{L^{\infty}(\Omega)}^{m} & \text{in } &\Omega, \\
\underline{u}(x) & = & g(x) & \text{on }  & \partial \Omega.
\end{array}
\right.
\end{equation}

The existence of such functions can be established using classical techniques. Moreover, it is straightforward to verify that \(\overline{u}\) and \(\underline{u}\) serve as a supersolution and a subsolution, respectively, to \eqref{Eq-Henon-type}. Consequently, by applying the Comparison Principle (Lemma \ref{Comp-Princ}), Perron's method ensures the existence of a viscosity solution in \(\mathrm{C}^0(\Omega)\) to \eqref{Eq-Henon-type}. More precisely, we obtain the following theorem.

\begin{theorem}[\bf Existence and Uniqueness]\label{exist_uniq}
Let $h \in \mathrm{C}^0([0,\infty))$ be a bounded, increasing, real-valued function satisfying $h(0) = 0$, and let $\mathfrak{a} \in \mathrm{C}^0(\overline{\Omega})$ be non-negative. Suppose there exist a viscosity subsolution $\underline{u} \in \mathrm{C}^0(\Omega) \cap \mathrm{C}^{0,1}(\Omega)$ and a viscosity supersolution $\overline{u} \in \mathrm{C}^0(\Omega) \cap \mathrm{C}^{0,1}(\Omega)$ to the equation
\begin{equation}
|\nabla u|^{\theta} \Delta_p^{\mathrm{N}} u+\mathscr{H}(\nabla u,x) = \mathfrak{a}(x)h(u) \quad \text{in} \quad \Omega,
\end{equation}
such that $\underline{u}= \overline{u} = g \in \mathrm{C}^0(\partial \Omega)$. Define the class of admissible functions:
\begin{equation}
\mathrm{S}_g(\Omega) := \left\{ v \in \mathrm{C}^0(\Omega) \mid v \text{ is a viscosity supersolution of } |\nabla u|^{\theta} \Delta_p^{\mathrm{N}} u+\mathscr{H}(\nabla u,x) = \mathfrak{a}(x)h(u) \right.
\end{equation}
\begin{equation*}
    \left. \text{ in } \Omega, \ \text{such that } \underline{u} \leq v \leq \overline{u} \text{ and } v = g \text{ on } \partial \Omega \right\}.
\end{equation*}
Then, the function
\begin{equation}
    u(x) := \inf_{v \in \mathcal{S}_g(\Omega)} v(x), \quad \text{for all } x \in \Omega,
\end{equation}
is the unique continuous (up to the boundary) viscosity solution to the boundary value problem
\begin{equation}
\left\{
\begin{array}{rclcl}
|\nabla u|^{\theta} \Delta_p^{\mathrm{N}} u +\mathscr{H}(\nabla u,x)& = & \mathfrak{a}(x)h(u) & \text{in } & \Omega, \\
 u(x) & = & g(x) & \text{on }& \partial \Omega.
\end{array}
\right.
\end{equation}
\end{theorem}

\medskip

We now establish higher regularity estimates for bounded viscosity solutions along free boundary points within this context.

\begin{theorem}[{\bf Higher Regularity Estimates}]\label{Higher_continuity}
Let $u \in \mathrm{C}^0(\mathrm{B}_1)$ be a non-negative, bounded viscosity solution to problem \eqref{Eq-Henon-type}. In the sublinear case, assume that \(\alpha\geq \frac{\sigma -m(2+\theta-\sigma)}{1+\theta-\sigma}\) and \(m\leq \frac{\sigma}{2+\theta-\sigma}\). Given $x_0 \in \mathrm{B}_{1/2} \cap \mathfrak{f}^{-1}(0)$, if $x_0 \in \mathrm{B}_{1/2} \cap \partial \{ u > 0 \}$, then
\begin{equation}\label{Higher Reg'}
u(x) \leq \mathrm{C} \cdot \|u\|_{L^{\infty}(\mathrm{B}_1)} |x - x_0|^{\frac{2+\theta+\alpha}{1+\theta - m}},
\end{equation}
for every $x \in \{u > 0\} \cap \mathrm{B}_{1/2}$, where $\mathrm{C} > 0$ depends only on universal parameters.
\end{theorem}

\begin{remark}
    As an immediate consequence of the Strong Maximum Principle (Theorem \ref{thm:SMP}), we obtain the following result: 
If $u$ is a nonnegative, bounded viscosity solution to \eqref{Eq-Henon-type} with $m = \gamma + 1$, then  either $u > 0$ or $u \equiv 0$ in $B_1$.
\end{remark}

\begin{example}
As an illustrative example, let \( p \in (1,  \infty) \) and \(\alpha \geq \frac{\sigma}{1+\theta-\sigma}\) be fixed constants. Consider the function defined by  
\[
u_{i}(x_1, \ldots, x_n) =  \mathfrak{c}_{p,\alpha,\theta} |x_i|^{\frac{2+\alpha+\theta}{1+\theta}}, \quad x \in \mathrm{B}_1,
\]
with 
\[
\mathfrak{c}_{p,\alpha,\theta} = \frac{(1+\theta)^{\frac{\theta}{1+\theta}} \left[(1+\alpha)(p-1)\right]^{\frac{1}{1+\theta}}}{2+\theta+\alpha},
\]
which satisfies the following equation in the viscosity sense:
\[
|\nabla u_{i}|^{\theta} \Delta^{\mathrm{N}}_p u_{i}(x) + \mathscr{H}(\nabla u_{i},x) = 2|x_i|^{\alpha} + x_{i} |x_{i}|^{1+\alpha} \lesssim |x|^{\alpha} \quad \text{in} \quad \mathrm{B}_1,
\]
where \(\mathfrak{B}(x) = \mathfrak{c}_{p,\alpha,\theta}^{-(1+\theta)} \left(\frac{2+\alpha+\theta}{1+\theta}\right)^{-(1+\theta)} e_{i}\) and  \(\varrho(x) = \mathfrak{c}_{p,\alpha,\theta}^{-\sigma} \left(\frac{2+\alpha+\theta}{1+\theta}\right)^{-\sigma} |x_{i}|^{\alpha - \frac{1+\alpha}{1+\theta} \sigma}\). 

Moreover, the singular (or critical) set is given by
\[
\mathrm{Sing}_u(\mathrm{B}_1) = \{x \in \mathrm{B}_1 : |\nabla u| = 0\} = \{x \in \mathrm{B}_1 : x_i = 0\}.
\]
In particular, we have \( u \in \mathrm{C}_{\text{loc}}^{1, \frac{1+\alpha}{1+\theta}}(\mathrm{B}_1) \).
\end{example}

In contrast with problem \eqref{Problem}, the following result rigorously characterizes the sharp asymptotic behavior with which non-negative viscosity solutions to \eqref{Eq-Henon-type} vanish away from their dead-core regions. This information is crucial in the analysis of several free boundary problems (see \cite{KKPS2000}, \cite{Lee-Shahg2003}, and \cite{Teix16}), and plays a pivotal role in establishing weak geometric properties.

\begin{theorem}[{\bf Non-degeneracy}]\label{NDHTE}
Let \(u \in \mathrm{C}^{0}(\mathrm{B}_{1})\) be a non-negative viscosity solution to \eqref{Eq-Henon-type}, and let \(x_{0} \in \overline{\{u > 0\}} \cap \mathrm{B}_{1/2}\). Suppose further that \(\mathfrak{f}(x) \geq \mathrm{c}_0 |x|^{\alpha}\) in \(\mathrm{B}_1\), and that \(\sigma > 1 + \theta\). Then, there exists a constant \(\mathrm{C}_{\mathrm{ND}} > 0\), depending only on \(n\), \(p\), \(\theta\), \(\sigma\), \(\mathfrak{c}_{0}\), \(\|\mathfrak{B}\|_{L^{\infty}(\mathrm{B}_{1};\mathbb{R}^{n})}\), and \(\|\varrho\|_{L^{\infty}(\mathrm{B}_{1})}\), such that for all \(r \in (0,1/2)\), we have
\[
\sup_{\partial \mathrm{B}_{r}(x_0)} u \geq \mathrm{C}_{\mathrm{ND}} r^{\frac{2+\theta+\alpha}{1+\theta - m}}.
\]   
\end{theorem}

\begin{remark}
By inspecting the proof of the result above, one can also derive non-degeneracy estimates in the sublinear case \(\theta<\sigma < 1 + \theta\), for instance, when \(\varrho \leq 0\).
\end{remark}

\section{Some motivation and literature review}\label{Section02}

\subsection*{Fully nonlinear models with Hamiltonian terms}

The mathematical literature concerning problems governed by fully nonlinear partial differential equations (PDEs) with Hamiltonian terms is relatively recent but has already produced substantial contributions to elliptic regularity theory.

In \cite{BirDem15}, Berindelli and Demengel initiated the study of the regularity properties of solutions to a class of fully nonlinear elliptic PDEs that include a sublinear first-order term of the form
\[
F(D^2u) + b(x)|\nabla u|^{\beta} = f(x) \quad \text{in} \quad \Omega \subset \mathbb{R}^n,
\]
where \( F: \mathrm{Sym}(n) \to \mathbb{R} \) satisfies the uniform ellipticity condition \eqref{Unif-Ellipticity}, and $\beta \in (0, 1)$.

The authors demonstrate that solutions to these PDEs exhibit local Hölder continuous gradients in the interior of the domain. Specifically, they establish the following estimate:
\[
\|u\|_{\mathrm{C}^{1,\gamma}(B_r(x_0))} \leq \mathrm{C}(\gamma) \left( \|u\|_{L^{\infty}(\Omega)} + \|b\|_{L^{\infty}(\Omega)}^{\frac{1}{1 - \beta}} + \|f\|_{L^{\infty}(\Omega)}^{\frac{1}{1+\alpha}} \right),
\]
provided that \( B_r(x_0) \subset \Omega \) for any \( r \in (0, 1) \), where \( \gamma \in (0, 1) \) depends on the ellipticity constants of \( F \), as well as on \( \|b\|_{L^{\infty}(\Omega)} \), \( \omega_b \) (modulus of continuity of $b$), and \( \beta \).

To obtain their results, the authors employ an ``improvement of flatness'' lemma: if a solution is approximately flat—meaning close to a linear profile—at a given scale, then it becomes flatter at smaller scales. Iterating this process yields the desired regularity. Their approach also relies on compactness arguments and builds upon prior work in the field, adapting existing techniques to accommodate the sublinear nature of the first-order term.

Recently, da Silva and Nornberg, in \cite{daSNorb21} (see the references therein for an extensive list of related works), established sharp regularity estimates for \( L^p \)-viscosity solutions of fully nonlinear second-order equations of the form
\[
F(x, Du, D^2 u) = f(x) \quad \text{in} \quad \Omega,
\]
where \( \Omega \subset \mathbb{R}^n \) is a bounded domain, and \( F \) denotes a uniformly elliptic operator satisfying appropriate growth conditions with respect to the gradient variable. The source term \( f \) belongs to a suitable Lebesgue space \( L^p(\Omega) \), with \( p \in (p_0, +\infty] \), where \( p_0 \) is the Escauriaza exponent.

In this setting, \( F: \Omega \times \mathbb{R}^n \times \mathrm{Sym}(n) \to \mathbb{R} \) is a measurable function, and \( H: \Omega \times \mathbb{R}^n \times \mathbb{R}^n \to \mathbb{R} \) represents the Hamiltonian term. The following structural conditions are imposed:
\begin{itemize}
\item[(A1)] \( F(\cdot, 0, 0) \equiv 0 \), and for any \( \xi, \eta \in \mathbb{R}^n \) and \( \mathrm{X}, \mathrm{Y} \in \mathrm{Sym}(n) \), the inequality
\[
\mathcal{M}_{\lambda, \Lambda}^-(\mathrm{X} - \mathrm{Y}) - H(x, \xi, \eta)
\leq F(x, \xi, \mathrm{X}) - F(x, \eta, \mathrm{Y}) \leq \mathcal{M}_{\lambda, \Lambda}^+(\mathrm{X} - \mathrm{Y}) + H(x, \xi, \eta)
\]
holds, where \( H \) takes one of the following two forms, for some \( q, \varrho \geq p \) with \( q, \varrho > n \), and \( \mathcal{M}^{\pm}_{\lambda, \Lambda}(\cdot) \) denote the \textit{Pucci extremal operators}, as defined in \eqref{Pucci}:
\begin{enumerate}
\item (\textbf{Sublinear regime}) For \( 0 < m < 1 \),
\[
H(x, \xi, \eta) = b(x) |\xi - \eta| + \mu(x) |\xi - \eta|^m;
\]

\item (\textbf{Superlinear/Quadratic regime}) For \( 1 < m \leq 2 \),
\[
H(x, \xi, \eta) = b(x) |\xi - \eta| + \mu(x) |\xi - \eta| \left( |\xi|^{m-1} + |\eta|^{m-1} \right);
\]
\end{enumerate}
where \( b \in L^{\varrho}_+(\Omega) \) and \( \mu \in L^q_+(\Omega) \), with \( q, \varrho \in (n, +\infty) \); additionally, \( \mu(x) \equiv \mu \geq 0 \) if \( m = 2 \).
\end{itemize}

It is important to emphasize that the authors’ methodology addresses both the sublinear and superlinear gradient regimes within a unified analytical framework.

In this context, the authors prove gradient estimates in \cite[Theorem 1.6]{daSNorb21} across all growth regimes, under suitable integrability conditions on the data and the coefficients of the operator.

\medskip

Concerning fully nonlinear models of degenerate type, Birindelli and Demengel in \cite{BirDem14} investigated the regularity of viscosity solutions to the following Dirichlet boundary value problem:
\begin{equation}\label{Eq01BD}
\left\{
\begin{array}{rcll}
|\nabla u|^\alpha \left( F(D^2 u)  + h(x) \cdot \nabla u \right)  &=& f(x) & \text{in } \Omega, \\
u(x) &=& \varphi(x) & \text{on } \partial \Omega,
\end{array}
\right.
\end{equation}
where \(\Omega \subset \mathbb{R}^n\) is a bounded domain of class \(\mathrm{C}^2\), \(\alpha \geq 0\), \(h \in C(\overline{\Omega}, \mathbb{R}^n)\), \(f \in C(\overline{\Omega})\), \(\varphi \in \mathrm{C}^{1,\beta_0}(\partial \Omega)\), and \(F\) is a uniformly elliptic operator; that is, there exist constants \(0 < \lambda \leq \Lambda\) such that for all symmetric matrices \(\mathrm{M}, \mathrm{N} \in \mathrm{Sym}(n)\) with \(\mathrm{N} \geq 0\),
\begin{equation}\label{Unif-Ellipticity}
\mathcal{M}^{-}_{\lambda, \Lambda}(\mathrm{N}) \leq F(\mathrm{M} + \mathrm{N}) - F(\mathrm{M}) \leq \mathcal{M}^{+}_{\lambda, \Lambda}(\mathrm{N}),
\end{equation}
where \(\mathcal{M}^{\pm}_{\lambda, \Lambda}\) denote the Pucci extremal operators.

Under these assumptions, the authors establish the existence of constants \(\mathrm{C} = \mathrm{C}(\beta)\) and \(\beta = \beta(\lambda, \Lambda, \|f\|_{L^{\infty}(\Omega)}, n, \Omega, \|h\|_{L^{\infty}(\Omega)}, \beta_0)\) such that any viscosity solution to problem \eqref{Eq01BD} satisfies the following regularity estimate:
\[
\|u\|_{\mathrm{C}^{1,\beta}(\overline{\Omega})} \leq \mathrm{C} \left( \|\varphi\|_{\mathrm{C}^{1,\beta_0}(\partial \Omega)} + \|u\|_{L^{\infty}(\Omega)} + \|f\|_{L^{\infty}(\Omega)}^{\frac{1}{1+\alpha}} \right).
\]

Subsequently, Birindelli and Demengel in \cite{BirDem16} addressed the boundary regularity of viscosity solutions to the following Dirichlet problem with a linear or sublinear gradient term:
\begin{equation}\label{Eq02BD}
\left\{
\begin{array}{rcll}
|\nabla u|^{\alpha} F(D^2 u) + |\nabla u|^{\beta} b(x) &=& f(x) & \text{in } \Omega, \\
u(x) &=& \varphi(x) & \text{on } \partial \Omega,
\end{array}
\right.
\end{equation}
where \(\Omega \subset \mathbb{R}^n\) is a bounded domain of class \(\mathrm{C}^2\), \(\alpha > -1\), \(\beta \in [0, \alpha + 1]\), \(b, f \in C(\overline{\Omega})\), \(\varphi \in \mathrm{C}^{1,\gamma_0}(\partial \Omega)\), and \(F\) is a uniformly elliptic operator. That is, there exist constants \(0 < \lambda < \Lambda\) such that for all \(\mathrm{M}, \mathrm{N} \in \mathrm{Sym}(n)\) with \(\mathrm{N} \geq 0\), the inequality
\[
\lambda \, \mathrm{tr}(\mathrm{N}) \leq F(\mathrm{M} + \mathrm{N}) - F(\mathrm{M}) \leq \Lambda \, \mathrm{tr}(\mathrm{N}),
\]
holds, with normalization \(F(0) = 0\).

In this context, the authors prove the existence of a constant \(\gamma \in (0, \gamma_0)\), depending only on \(\lambda\), \(\Lambda\), \(\alpha\), and \(\beta\), and a constant \(\mathrm{C} = \mathrm{C}(\gamma)\), such that any bounded viscosity solution to problem \eqref{Eq02BD} belongs to the space \(\mathrm{C}^{1,\gamma}(\overline{\Omega})\). Moreover, in the sublinear regime \(\beta < \alpha + 1\), the following estimate holds:
\[
\|u\|_{\mathrm{C}^{1,\gamma}(\overline{\Omega})} \leq \mathrm{C} \left( \|u\|_{L^{\infty}(\Omega)} + \|b\|_{L^{\infty}(\Omega)}^{\frac{1}{1+\alpha - \beta}} + \|f\|_{L^{\infty}(\Omega)}^{\frac{1}{1+\alpha}} + \|\varphi\|_{\mathrm{C}^{1,\gamma_0}(\partial \Omega)} \right).
\]
Additionally, in the linear case \(\beta = \alpha + 1\), the Hölder exponent \(\gamma\) also depends on the norm \(\|b\|_{L^{\infty}(\Omega)}\).

\medskip

In the work \cite{BirDemLeoni19}, Birindelli and Demengel addressed the existence and uniqueness of viscosity solutions to the Dirichlet problem
\begin{equation}\label{Eq03BD}
\left\{
\begin{array}{rcrcl}
-F(\nabla u, D^2 u) +  b(x)|\nabla u|^{\beta} + \lambda^{\ast}|u|^{\alpha}u & = & f(x) & \text{in } &\Omega, \\
u(x) & = & \varphi(x) & \text{on } & \partial \Omega,
\end{array}
\right.
\end{equation}
where \(\Omega \subset \mathbb{R}^n\) is a bounded open set of class \(\mathrm{C}^2\). Moreover, the operator \(F\) satisfies the following structural assumptions:
\begin{itemize}
    \item[(1)] \(F : \mathbb{R}^n \setminus \{0\} \times \text{Sym}(n) \to \mathbb{R}\) is a continuous function.
    
    \item[(2)] \(F(\xi, \mathrm{M})\) is homogeneous of degree \(\alpha > -1\) with respect to the variable \(\xi\), and positively homogeneous of degree one with respect to \(\mathrm{M}\). Furthermore, there exist constants \(\Lambda \geq \lambda > 0\) such that
    \begin{equation}\label{eq:ellipticity}
        \lambda|\xi|^{\alpha} \mathrm{tr}(\mathrm{N}) \leq F(\xi, \mathrm{M} + \mathrm{N}) - F(\xi, \mathrm{M}) \leq \Lambda|\xi|^{\alpha} \mathrm{tr}(\mathrm{N})
    \end{equation}
    for all \(\mathrm{M}, \mathrm{N} \in \text{Sym}(n)\) with \(\mathrm{N} \geq 0\). In addition, there exists a constant \(\mathrm{c}_0 > 0\) such that
    \begin{equation}\label{eq:continuity}
        |F(\xi_1, \mathrm{M}) - F(\xi_2, \mathrm{M})| \leq \mathrm{c}_0\|\mathrm{M}\| \cdot \big||\xi_1|^{\alpha} - |\xi_2|^{\alpha} \big|
    \end{equation}
    for all \(\xi_1, \xi_2 \in \mathbb{R}^n \setminus \{0\}\) and \(\mathrm{M} \in \text{Sym}(n)\).
\end{itemize}

Additionally, they assume that \(\lambda^{\ast} > 0\), the first-order coefficient \(b\) is Lipschitz continuous, the source term \(f\) is bounded and continuous, and the boundary data \(\varphi\) is also Lipschitz continuous.

Under these assumptions, they proved that the unique solution to \eqref{Eq03BD} is Lipschitz continuous.

\medskip

Finally, in the manuscript \cite{daSLR21}, Da Silva \textit{et al.} investigated reaction-diffusion problems governed by the second-order fully nonlinear elliptic equation:
\begin{equation}\label{Eq-DC-GeneralFN}
F_0(x, \nabla u, D^2 u) + |\nabla u|^{\gamma} \langle \vec{b}(x), \nabla u \rangle = \lambda_0(x) u^{\mu} \chi_{\{u > 0\}}(x), \quad \text{in } \Omega,
\end{equation}
where \(\Omega \subset \mathbb{R}^n\) is a smooth, open, and bounded domain, \( g \geq 0 \), \( g \in \mathrm{C}^0(\partial \Omega) \), \(\vec{b} \in \mathrm{C}^0(\overline{\Omega}, \mathbb{R}^n)\), \( 0 \leq \mu < \gamma + 1 \) is the absorption exponent, \(0<\lambda_0 \in \mathrm{C}^0(\overline{\Omega})\), and \( F_0: \Omega \times (\mathbb{R}^n \setminus \{0\}) \times \text{Sym}(n) \to \mathbb{R} \) is a second-order fully nonlinear elliptic operator with measurable coefficients satisfying certain ellipticity and homogeneity conditions. Specifically,
\[
|\vec{\xi}|^{\gamma} \mathcal{M}^{-}_{\lambda, \Lambda}(\mathrm{X} - \mathrm{Y}) \leq F_0(x, \vec{\xi}, \mathrm{X}) - F_0(x, \vec{\xi}, \mathrm{Y}) \leq |\vec{\xi}|^{\gamma} \mathcal{M}^{+}_{\lambda, \Lambda}(\mathrm{X} - \mathrm{Y}) \quad (\text{with } \gamma > -1)
\]
for all \(\mathrm{X}, \mathrm{Y} \in \text{Sym}(n)\) with \(\mathrm{X} \geq \mathrm{Y}\).

The authors establish the sharp asymptotic behavior that characterizes the dead core sets of non-negative viscosity solutions to \eqref{Eq-DC-GeneralFN}. Furthermore, they derive a precise regularity estimate at free boundary points, namely \(\mathrm{C}_{\text{loc}}^{\frac{\gamma + 2}{\gamma + 1 - \mu}}\). They also determine the exact decay rate of the gradient at interior free boundary points and prove a sharp Liouville-type theorem.

\subsection*{Regularity estimates for quasilinear models}

For \( 1< p<  \infty \), consider the \( p \)-Laplace equation
\begin{equation}\label{Eq-p-Laplace}
  -\Delta_p u \defeq - \text{div} \left( |\nabla u|^{p-2} \nabla u \right) = 0 \quad \text{in} \quad \Omega,
\end{equation}
where \( \Omega \subset \mathbb{R}^n \) is a bounded open domain. Weak solutions to the quasilinear equation \eqref{Eq-p-Laplace} are commonly referred to as \( p \)-harmonic functions. Observe that for \( p > 2 \), equation \eqref{Eq-p-Laplace} is degenerate elliptic, whereas for \( 1 < p < 2 \), it is singular, since the ellipticity modulus \(\mathbf{\Phi}(\nabla u) = |\nabla u|^{p-2}\) degenerates (or blows up) at points where \( |\nabla u| = 0 \).

With respect to regularity estimates, in the late 1960s, Ural’tseva \cite{Ural68} proved that for \( p > 2 \), weak solutions of equation \eqref{Eq-p-Laplace} possess Hölder continuous first derivatives in the interior of \(\Omega\). In the early 1980s, Lewis \cite{Lewis83} and DiBenedetto \cite{DiBe} established analogous results valid for the singular case \( 1 < p < 2 \). However, in general, weak solutions exhibit no higher regularity than \( C_{\text{loc}}^{1, \alpha} \) for some unknown exponent \( \alpha \in (0, 1) \).

Beyond these foundational contributions, it is important to highlight the sharp regularity results in dimension \( n = 2 \). More precisely, Aronsson \cite{Arons89} and Iwaniec–Manfredi \cite{IwaManf89} showed that a \( p \)-harmonic function belongs to the spaces \( \mathrm{C}^{k,\alpha}_{\mathrm{loc}} \cap W^{2+k, q}_{\mathrm{loc}} \), where the integer \( k \geq 1 \) and the Hölder exponent \( \alpha \in (0, 1) \) are determined by the formula:
\begin{equation}\label{SharpExp}
k + \alpha^{\sharp}_p = \frac{1}{6} \left( 7 + \frac{1}{p - 1} + \sqrt{1 + \frac{14}{p - 1} + \frac{1}{(p - 1)^2}} \right), \quad \text{with}  \quad 1 \leq q(p) < \frac{2}{2-\alpha}.
\end{equation}

The proof crucially relies on techniques introduced in the foundational survey by Bojarski and Iwaniec \cite{BojIwa84}. A key element in their approach is the use of the hodograph transformation, which converts the nonlinear \( p \)-harmonic equation into a linear first-order elliptic system. This auxiliary system is then solved using Fourier series methods. A meticulous analysis of the Fourier expansion of the solution to the transformed system ultimately yields the desired regularity result. Notably, they demonstrate that the exponent given in \eqref{SharpExp} is optimal. For further developments closely related to this methodology, we refer the reader to the works of Manfredi \cite{Manf88} and Aronsson–Lindqvist \cite{AronLind88}, which build upon and complement the ideas presented in \cite{BojIwa84}.

\medskip

More recently, in the influential work \cite{ATU17}, the authors studied the \( p \)-Poisson equation
\begin{equation}\label{Eqp-Poisson}
  -\Delta_p u = f(x), \quad \text{in} \quad \mathrm{B}_1 \subseteq \mathbb{R}^2,
\end{equation}
for \( p > 2 \), assuming additionally that \( f \in L^\infty(B_1) \). Within this framework, they resolved a long-standing open problem: the proof of the \( \mathrm{C}^{p'} \)-regularity conjecture in the plane, where this regularity exponent is shown to be optimal.
\begin{theorem}[{\cite{ATU17}}]
Let \( \mathrm{B}_1 \subset \mathbb{R}^2 \), and let \( u \in W^{1,p}(\mathrm{B}_1) \) be a weak solution of
\[
-\Delta_p u = f(x), \quad p > 2,
\]
with \( f \in L^\infty(\mathrm{B}_1) \). Then, \( u \in \mathrm{C}^{p^{\prime}}(\mathrm{B}_{1/2}) \) and
\[
\|u\|_{\mathrm{C}^{p^{\prime}}(\mathrm{B}_{1/2})} \leq \mathrm{C}_p \left( \|u\|_{L^{\infty}(\mathrm{B}_1)} + \|f\|^{\frac{1}{p-1}}_{L^\infty(\mathrm{B}_1)}\right).
\]
\end{theorem}

It is important to note that in the statement above, the constant \( p^{\prime} \) denotes the well-known H\"{o}lder conjugate of \( p \), i.e.,
\[
p + p^{\prime} = pp^{\prime} \quad \Leftrightarrow \quad p^{\prime} = \frac{1}{1 - \frac{1}{p}} = 1 + \frac{1}{p - 1}.
\]

A key element in \cite{ATU17} is the conditional result that the conjecture holds provided that \( p \)-harmonic functions—i.e., weak solutions to \( -\Delta_p \mathfrak{h} = 0 \)—are locally of class \( \mathrm{C}^{1, \alpha_p} \) for some \( \alpha_p > \frac{1}{p - 1} \). While this regularity remains an open and challenging problem in higher dimensions (cf. \cite{ATU18}), it is known to hold in the plane, thus yielding a complete resolution of the conjecture in two dimensions.

This key estimate follows from a seminal result by Baernstein II and Kovalev \cite{BaerKova05}, which leverages the fact that the complex-valued gradient \( \phi = \frac{\partial u}{\partial z} \) of a \( p \)-harmonic function in the plane is a \( \mathbf{K} \)-quasiregular gradient mapping, where \( \mathbf{K} = \mathbf{K}(p) = p - 1 \). The complex-valued gradient \( \phi = u_x - i u_y \in W^{1,2}(\mathrm{B}_1) \) satisfies the following first-order system in the hodograph (complex) plane:
\[
\phi_{\overline{z}} = \left( \frac{1}{p} - \frac{1}{2} \right) \left( \frac{\phi_z \overline{\phi}}{\phi} + \frac{\overline{\phi}_z \phi}{\overline{\phi}} \right), \quad \text{where} \quad z = x + i y.
\]
By invoking the classical Hölder regularity results for quasiregular mappings (see \cite[Section 5]{BaerKova05} and \cite{LL}), it follows that a weak solution to the \( p \)-Laplacian belongs to the class \( \mathrm{C}^{1, \alpha^{\ast}_p}_{\mathrm{loc}} \), with exponent
\[
\alpha^{\ast}_p = \frac{1}{2p} \left( -3 - \frac{1}{p - 1} + \sqrt{33 + \frac{30}{p - 1} + \frac{1}{(p - 1)^2}} \right).
\]
Moreover, the following estimate holds:
\[
[\phi]_{\mathrm{C}^{0, \alpha^{\ast}_p}(\mathrm{B}_{1/2})} \leq \mathrm{C}_p \|u\|_{L^{\infty}(\mathrm{B}_1)}.
\]

We conclude this subsection by summarizing the result above in the form of the following proposition.

\begin{proposition}[{\cite{ATU17} and \cite{BaerKova05}}]\label{PropB-K}
For any \( p > 2 \), there exists a constant \( 0 < \tau_0 < \frac{p-2}{p-1} \) such that \( p \)-harmonic functions in \( \mathrm{B}_1 \subset \mathbb{R}^2 \) are locally of class \( \mathrm{C}^{p^{\prime} + \tau_0} \). Furthermore, if \( u \in W^{1,p}(\mathrm{B}_1) \cap \mathrm{C}^0(\mathrm{B}_1) \) is \( p \)-harmonic in the unit disk \( \mathrm{B}_1 \subset \mathbb{R}^2 \), then there exists a constant \( \mathrm{C}_p > 0 \), depending only on \( p \), such that
\[
[\nabla u]_{\mathrm{C}^{0, \frac{1}{p-1} + \tau_0}(\mathrm{B}_{1/2})} \leq \mathrm{C}_p \| u \|_{L^\infty(\mathrm{B}_1)}.
\]
\end{proposition}

In conclusion, it is important to emphasize the following inequality:
\[
\alpha_p^{\sharp} > \alpha^{\ast}_p > \frac{1}{p-1} \quad \text{for any} \quad p > 2.
\]

\subsection*{Higher estimates for H\'{e}non-type problems with strong absorption}

Recently, Teixeira in \cite{Tei22} derived regularity estimates for interior stationary points of solutions to $p$-degenerate elliptic equations in inhomogeneous media:

\[
\mathrm{div}\,\mathfrak{a}(x, \nabla u) = f(x, u, \nabla u), \quad \text{with} \quad  |f(x, s, \xi)| \lesssim \mathrm{c}_0 |x|^{\alpha} |s|^{m} \min\{1, |\xi|^{\kappa}\},
\]
for some $\mathrm{c}_0 > 0$, $\alpha, \kappa \geq 0$, and $0 \leq m < p - 1$. In this setting, the vector field $\mathfrak{a}: \mathrm{B}_1 \times \mathbb{R}^n \to \mathbb{R}^n$ is continuously differentiable in the gradient variable and satisfies the following structural conditions:

\begin{equation} \label{condestr}
\left\{
\begin{array}{rclcl}
|\mathfrak{a}(x,\xi)| + |\partial_{\xi}\mathfrak{a}(x,\xi)||\xi| & \leq & \Lambda |\xi|^{p-1}, & & \\
\lambda |\xi|^{p-2}|\eta|^2 & \leq & \langle \partial_{\xi}\mathfrak{a}(x,\xi)\eta, \eta \rangle, & & \\
\displaystyle \sup_{\genfrac{}{}{0pt}{}{x, y \in \mathrm{B}_1}{x \ne y, \,\,\,|\xi| \ne 0 }} \frac{|\mathfrak{a}(x,\xi)-\mathfrak{a}(y,\xi)|}{\omega(|x-y|)|\xi|^{p-1}} & \leq & \mathfrak{L}_0 < \infty, & &
\end{array}
\right.
\end{equation}
for $p \geq 2$, with constants $\lambda, \Lambda, \mathfrak{L}_0 > 0$ and a modulus of continuity $\omega: [0, \infty) \to [0, \infty)$ satisfying
\[
\int_{0}^{2} \frac{\omega^{\frac{2}{p}}(t)}{t^{1+\gamma}}\,dt < \infty \quad \text{for some } \gamma \in [0, 1).
\]

In this context, Teixeira established a quantitative non-degeneracy estimate, demonstrating that solutions cannot be smoother than $\mathrm{C}^{p'}$ at stationary points (see \cite[Proposition 5]{Tei22}). Moreover, at critical points where the source term vanishes, sharp higher-order regularity estimates are obtained, reflecting the vanishing rate of the source (see \cite[Theorem 3]{Tei22}).
\medskip

Subsequently, we consider quasilinear elliptic equations of the form:
\begin{equation}\label{EqMatukuma}
\mathrm{div}\left(\mathfrak{a}(|x|) |\nabla u|^{p-2} \nabla u\right) = \mathfrak{h}(|x|) f(u) \quad \text{in } \quad \Omega \subset \mathbb{R}^n, \quad  p > 1,
\end{equation}
where $\mathfrak{a}, \mathfrak{h} : \mathbb{R}^+ \to \mathbb{R}^+$ are radial profiles, with $\mathfrak{a} \in \mathrm{C}^1(\mathbb{R}^+)$ and $\mathfrak{h} \in \mathrm{C}^0(\mathbb{R}^+)$. A notable example is the celebrated Matukuma equation (also known as the Batt–Faltenbacher–Horst equation; see \cite{BFH86}), which serves as a prototype for \eqref{EqMatukuma}.

The function \(f\) is assumed to satisfy the following structural conditions:
\begin{itemize}
    \item[\textbf{(F1)}] $f \in \mathrm{C}^0(\mathbb{R})$;
    \item[\textbf{(F2)}] $f$ is non-decreasing on $\mathbb{R}$, and $f(t) > 0$ if and only if $t > 0$.
\end{itemize}

As a specific example, consider the model equation:
\begin{equation}\label{ModelEq}
    \mathrm{div} \left( |x|^k | \nabla u|^{p-2} \nabla u \right) = |x|^{\alpha} f(u),
\end{equation}
where $f(u) = u_{+}^{m}$ and the parameters satisfy $m + 1 < k - \alpha < p$.

In this setting, da Silva \textit{et al.} \cite{daSdosPRS} established the following sharp estimate for weak solutions to \eqref{ModelEq}:
\[
\sup_{x \in \mathrm{B}_r(x_0)} u(x) \leq \mathrm{C} r^{1 + \frac{1 + \alpha + m - k}{p - 1 - m}},
\]
where $\mathrm{C} > 0$ is a universal constant, $x_0 \in \mathrm{B}_1$ is a free boundary point of \( u \), and the nonlinearity satisfies $f(|x|, u) \lesssim |x - x_0|^{\alpha} u_+^m$, with $\alpha + 1 + m > k$.
\medskip

Recently, Bezerra J\'{u}nior \textit{et al.} \cite{BeJDaSNS2024} established sharp and improved regularity estimates for non-negative viscosity solutions of H\'{e}non-type elliptic equations governed by the infinity-Laplacian under a strong absorption condition:
\begin{equation}\label{pobst}
\Delta_{\infty} u(x) = f(|x|, u(x)) \quad \text{in} \quad \mathrm{B}_1,
\end{equation}
where, for all $(x, t) \in \mathrm{B}_1 \times \mathfrak{I}$ (with $\mathfrak{I} \subset \mathbb{R}$ an interval) and $r, s \in (0, 1)$, the authors assume the existence of a universal constant $\mathrm{c}_n > 0$ and a function $f_0 \in L^{\infty}(\mathrm{B}_1)$ such that
\begin{equation}\label{EqHomog-f}
|f(r|x|, s t)| \leq \mathrm{c}_n r^{\alpha} s^m |f_0(x)| \quad \text{for} \quad 0 \leq m < 3 \quad  \text{and} \quad \alpha \in \left[0, \infty\right).
\end{equation}
A representative toy model for \eqref{pobst} is the H\'{e}non-type equation with strong absorption:
\[
\Delta_{\infty} u(x) = c_0 |x|^{\alpha} u_{+}^{m}(x) \quad \text{in} \quad \mathrm{B}_1 \quad \text{with} \quad c_0>0.
\]

In this context, the authors derived the following sharp regularity estimate along the free boundary:
\[
u(x) \leq \mathrm{C}\cdot \|u\|_{L^{\infty}(\mathrm{B}_1)}|x - x_0|^{\frac{4 + \alpha}{3 - m}},
\]
where $\mathrm{C} > 0$ depends only on universal parameters.

Additionally, they established the following non-degeneracy estimate:
\[
\sup_{\partial \mathrm{B}_r(x_0)} u(x) \geq \left(\frac{(3 - m)^4}{(4 + \alpha)^3(1 + \alpha + m)}\right)^{\frac{1}{3 - m}} \cdot r^{\frac{4 + \alpha}{3 - m}}.
\]

Moreover, our manuscript is motivated by \cite[Theorem 1]{NSST23}, in which the authors establish superior regularity properties for solutions to fully nonlinear elliptic models of the form
\[
F(x, D^2u) = f(x, u, Du) \lesssim q(x)|u|^m \min \{1, |Du|^{\gamma}\}, \quad (m, \gamma \geq 0),
\]
at interior critical points, where $q \in L^p(\mathrm{B}_1)$ is a non-negative function and $p > n$. The key innovation of these estimates lies in their capacity to confer smoothness properties that surpass the intrinsic regularity constraints imposed by the heterogeneity of the problem.

\medskip

Recently, in \cite{ADaSS25}, Alcantara \textit{et al.} investigated the regularity properties of solutions to a class of quasilinear elliptic partial differential equations (PDEs) characterized by strong absorption and a non-divergence structure. Specifically, the authors consider equations of the form:
\[
|\nabla u(x)|^{\gamma} \Delta_p^{\mathrm{N}} u(x) = f(x, u) \quad \text{in} \quad B_1,
\]
where \( \gamma > -1 \), \( p \in (1, \infty) \), and \( f(x, u) \lesssim \mathfrak{a}(x) u_{+}^m \), with \( m \in [0, \gamma + 1) \). 

This setting accommodates the existence of "plateau regions," i.e., subsets where the non-negative solution \( u \) vanishes identically. The authors derive refined geometric regularity estimates in the space $\mathrm{C}_{\text{loc}}^{\frac{\gamma+2}{\gamma+1-m}}$ along the free boundary $\partial \{u>0\}\cap B_{1/2}$. In addition, they establish non-degeneracy results and explore measure-theoretic aspects of the solutions, offering deeper insights into the geometric structure of the free boundary. A sharp Liouville-type theorem is also proven for entire solutions exhibiting controlled growth at infinity, contributing to the broader understanding of the global behavior of such solutions. This work extends the regularity theory for quasilinear elliptic equations by addressing the case of strong absorption in the non-divergence setting, with particular attention to behavior near the free boundary. The results have significant implications for the qualitative analysis of solutions to such PDEs and enrich the theoretical framework surrounding nonlinear elliptic problems with absorption terms.

\medskip

Below, we summarize the sharp regularity estimates available in the literature for problems closely related to \eqref{Eq-Henon-type}:

\begin{table}[h]
\centering
\resizebox{\textwidth}{!}{
 \begin{tabular}{c|c|c|c}
{\bf Model PDE} & {\bf Structural assumptions} & {\bf Sharp regularity estimates} & \textbf{References} \\
\hline
$\mathrm{div} \left( |x|^k | \nabla u|^{p-2} \nabla u \right) = |x|^{\alpha}u_+^m$ & $0 \leq m < p-1$, \,\, $\alpha + 1 + m > k$, \,\, $p > 1$ & $C_\mathrm{loc}^{\frac{p+\alpha-k}{p-1-m}}$ & \cite{daSdosPRS} \\
\hline
$\Delta_p u = f(x, u) \lesssim \mathrm{c}_0|x|^{\alpha}|u|^{m}\min\{1, |\nabla u|^{\kappa}\}$ & $\alpha, \kappa \geq 0$, \,\, $0 \leq m < p-1$ (for $p > 2$) & $C_\mathrm{loc}^{1, \min\left\{\alpha_\mathrm{H}, \frac{\alpha+1+\gamma\kappa}{p-1-m}\right\}^{-}}$ & \cite{Tei22} \\
\hline
$ F(x, D^2u) = f(x, u, Du) \lesssim q(x)|u|^m \min \{1, |Du|^{\gamma}\}$ & $F$ \text{ uniformly elliptic},\,\,\,$g \in L^p(B_1)$, \,\,$p>n$,\,\,$m, \gamma\geq 0$ & $C_
{\text{loc}}^{1 + \epsilon_{m, n, p, \gamma}}$ & \cite{NSST23} \\
\hline
$ \Delta_{\infty} u(x) = \mathcal{G}(x, u) \lesssim |x|^{\alpha} u_{+}^m $ & $\alpha\geq 0 \,\,\,\text{and}\,\,\, 0\leq m<3$  & $C_
{\text{loc}}^{\frac{4+\alpha}{3-m}}$ & \cite{BeJDaSNS2024}
\\
\hline
$ |\nabla u|^{\gamma}\Delta^{\mathrm{N}}_{p} u(x) = f(x, u) \lesssim |x|^{\alpha} u_{+}^m $ & $\alpha\geq 0$, $0 \leq m<\gamma+1$,\,\, $p \in (1, \infty)$ & $C_
{\text{loc}}^{\frac{\gamma+2+\alpha}{\gamma+1-m}}$ & \cite{ADaSS25}
\\
\hline
$ |\nabla u|^{\theta}\Delta^{\mathrm{N}}_{p} u(x) + \mathscr{H}(\nabla u, x)= f(x, u) \lesssim |x|^{\alpha} u_{+}^m $ & $(\mathbf{H1})-(\mathbf{H4})$, $0 \leq m<\theta+1$,\,\, $\alpha \geq 0$ & $C_
{\text{loc}}^{\frac{\theta+2+\alpha}{\theta+1-m}}$ & Theorem \ref{Higher_continuity}
\end{tabular}}
\caption{Sharp regularity estimates for elliptic problems related to \eqref{Eq-Henon-type}.}
\end{table}

This work aims to improve and generalize some of these results by employing novel strategies and analytical techniques.


\section{Preliminaries}

In this section, we introduce the notion of viscosity solutions, which will serve as the analytical framework for our study. We also present some auxiliary results essential for the development of our main arguments. 

Throughout this work, we denote by \( B_r = B_r(0) \) the open ball of radius \( r \) centered at the origin.

We begin by defining the H\"older semi-norms. Given \( \alpha \in (0, 1] \) and \( \Omega \subset \mathbb{R}^n \), we set
\[
[u]_{\mathrm{C}^{0, \alpha}( \Omega)} = \sup_{\substack{x, y \in \Omega \\ x \neq y}} \frac{|u(x) - u(y)|}{|x - y|^{\alpha}},
\]
and
\[
[u]_{\mathrm{C}^{1,\alpha}(\Omega)} = \sup_{\genfrac{}{}{0pt}{}{x\in\Omega}{\rho>0}} \inf_{\genfrac{}{}{0pt}{}{\vec{q}\in\mathbb{R}^{n}}{c \in \mathbb{R}}} \sup_{z \in \mathrm{B}_\rho(x) \cap \Omega} \frac{\left| u(z) - \vec{q} \cdot z - c \right|}{\rho^{1+\alpha} }.
\]

The normalized \( p \)-Laplacian operator, except when \( p = 2 \), is not well-defined at points where \( |\nabla u| = 0 \), as it becomes discontinuous—even if \( u \) is smooth. This issue can be resolved by employing the framework of viscosity solutions (see \cite{Attou18}), using upper and lower semicontinuous envelopes, which serve as a relaxed notion of the operator.

Given a symmetric matrix \( \mathrm{X} \in \mathrm{Sym}(n) \), we denote its maximal and minimal eigenvalues by
\[
\lambda_{\max}(\mathrm{X}) \defeq \max_{\genfrac{}{}{0pt}{}{\xi \in \mathbb{R}^n}{\|\xi\|=1}} \langle \mathrm{X}\xi, \xi \rangle \quad \text{and} \quad \lambda_{\min}(\mathrm{X}) \defeq \min_{\genfrac{}{}{0pt}{}{\xi \in \mathbb{R}^n}{\|\xi\|=1}} \langle \mathrm{X}\xi, \xi \rangle.
\]

We now define viscosity solutions.

\begin{definition}[{\bf Viscosity Solutions}]\label{DefViscSol}
An upper semicontinuous function \( u \in \mathrm{C}^0(\Omega) \) is called a viscosity subsolution (respectively, supersolution) of
\begin{equation}\label{Vis1}
     |\nabla u|^{\theta}\Delta_{p}^{\mathrm{N}} u + \mathscr{H}(\nabla u, x) = f_0(x, u(x)), \quad \text{with} \quad f_0 \in \mathrm{C}^0(\Omega \times \mathbb{R}_+),
\end{equation}
if, whenever \( \phi \in \mathrm{C}^{2}(\Omega) \) and \( u - \phi \) attains a local maximum (respectively, minimum) at \( x_{0} \in \Omega \), the following conditions hold:

\begin{enumerate}
    \item If \( |\nabla \phi(x_{0})| \neq 0 \), then
    \[
      |\nabla \phi(x_0)|^{\theta}\Delta_{p}^{\mathrm{N}} \phi(x_{0}) + \mathscr{H}(\nabla \phi(x_0), x_0) \geq f_0\big(x_{0}, \phi(x_0)\big)
    \]
    \[
      \text{(respectively, }\, |\nabla \phi(x_0)|^{\theta}\Delta_{p}^{\mathrm{N}} \phi(x_{0}) + \mathscr{H}(\nabla \phi(x_0), x_0) \leq f_0\big(x_{0}, \phi(x_0)\big)\text{)}.
    \]

    \item If \( |\nabla \phi(x_{0})| = 0 \), then:
    \begin{enumerate}
        \item For \( p \geq 2 \),
        \[
        \Delta \phi(x_{0}) + (p - 2) \lambda_{\max}\big(D^{2} \phi(x_{0})\big) \geq f_0\big(x_{0}, \phi(x_0)\big)
        \]
        \[
        \text{(respectively, }\, \Delta \phi(x_{0}) + (p - 2) \lambda_{\min}\big(D^{2} \phi(x_{0})\big) \leq f_0\big(x_{0}, \phi(x_0)\big)\text{)}.
        \]
        
        \item For \( 1 < p < 2 \),
        \[
        \Delta \phi(x_{0}) + (p - 2) \lambda_{\min}\big(D^{2} \phi(x_{0})\big) \geq f_0\big(x_{0}, \phi(x_0)\big)
        \]
        \[
        \text{(respectively, }\, \Delta \phi(x_{0}) + (p - 2) \lambda_{\max}\big(D^{2} \phi(x_{0})\big) \leq f_0\big(x_{0}, \phi(x_0)\big)\text{)}.
        \]
    \end{enumerate}
\end{enumerate}

Finally, we say that \( u \) is a viscosity solution to \eqref{Vis1} if it satisfies the conditions for both subsolutions and supersolutions.
\end{definition}

The next lemma plays a key role in the upcoming section concerning the demonstration of the approximation scheme (see Lemma \ref{approxlemma}). It can be regarded as a cutting lemma.

\begin{lemma}[{\bf Cutting Lemma, \cite[Lemma 2.6]{Attou18}}]\label{CT-Lemma}
Let \( \theta > -1 \) and \( p > 1 \). Assume that \( w \) is a viscosity solution of  
\[
- |D w + \vec{q}|^{\theta} \left(\Delta w - (p - 2) \frac{ D^2 w (D w + \vec{q}), (D w + \vec{q}) }{|D w + \vec{q}|^2}\right) = 0.
\]
Then \( w \) is a viscosity solution of  
\[
-\Delta^{\mathrm{N}}_{p} (w + \vec{q}\cdot x) = - \Delta w - (p - 2) \frac{ D^2 w (D w + \vec{q}), (D w + \vec{q}) }{|D w + \vec{q}|^2} = 0.
\]
\end{lemma}

Moreover, for homogeneous problems, we have the following regularity result.

\begin{lemma}[{\bf \cite[Lemma 3.4]{Attou18}}]\label{Holder-Cont-hom}
Let \( p \in (1, \infty) \). Assume that \( \vec{q} \in \mathbb{R}^{n} \), and let \( w \) be a normalized viscosity solution to the problem
\[
\Delta w - (p - 2) \frac{ D^2 w (D w + \vec{q}), (D w + \vec{q}) }{|D w + \vec{q}|^2} = 0 \quad \text{in } \mathrm{B}_{1}.
\]
Then, for all \( r \in (0, 1/2] \), there exist \( \alpha^{\prime} \in (0, 1) \) and a constant \( \mathrm{C} = \mathrm{C}(p, n, \alpha^{\prime}) > 0 \) such that
\[
[w]_{\mathrm{C}^{1,\alpha^{\prime}}(\mathrm{B}_{r})} \leq \mathrm{C}.
\]
\end{lemma}

Finally, we recall the notion of superjets and subjets, as introduced by Crandall, Ishii, and Lions in \cite{CIL}.

\begin{definition}
The second-order superjet of \( u \) at \( x_0 \in \Omega \) is defined as
\[
\mathcal{J}^{2,+}_{\Omega} u(x_0) = \left\{ (\nabla \phi(x_0), D^2 \phi(x_0)) : \phi \in \mathrm{C}^2 \text{ and } u - \phi \text{ has a local maximum at } x_0 \right\}.
\]
The closure of the superjet is given by
\begin{eqnarray*}
\overline{\mathcal{J}}^{2,+}_{\Omega} u(x_0) = \left\{ (\vec{\xi}, \mathrm{X}) \in \mathbb{R}^n \times \text{Sym}(n): \exists\,\, (\vec{\xi}_k, \mathrm{X}_k) \in \mathcal{J}^{2,+}_{\Omega} u(x_k) \text{ such that}\right.\\ 
\left.(x_k, u(x_k), \vec{\xi}_k, \mathrm{X}_k) \to (x_0, u(x_0), \vec{\xi}, \mathrm{X}) \right\}.
\end{eqnarray*}
Similarly, we define the second-order subjet and its closure.
\end{definition}

With these definitions in place, we are now in a position to present the Ishii–Lions lemma, a fundamental tool for establishing compactness and existence results in the theory of partial differential equations. For further details, we refer the reader to \cite[Theorem 3.2]{CIL}.

\begin{lemma}[\bf Ishii-Lions Lemma]\label{IshiiLions}
Let $u_i$ be a upper semicontinuous function in $\mathrm{B}_1$ for $i=1,\ldots,k$. Let $\varphi$ be defined on $(\mathrm{B}_1)^{k}$ and such that the function 
\[
(x_1,\dots,x_k) \to \varphi(x_1,\dots,x_k) 
\]
is twice continuously differentiable in $(x_1,\dots,x_k) \in (\mathrm{B}_1)^k$. Suppose that
\[
w(x_1,\dots,x_k) := u_1(x_1) + \cdots + u_k(x_k) - \varphi(x_1,\dots,x_k)
\]
attains a local maximum at $(\bar{x}_1,\dots,\bar{x}_k) \in (\mathrm{B}_1)^{k}$. Assume, moreover, that there exists an $r > 0$ such that for every $\mathrm{M}_{\star} > 0$ there is a constant $\mathrm{C}_{\star}$ such that for $i=1,\ldots,k$,
\[
b_i \leq \mathrm{C}_{\star} \quad \text{whenever } (\vec{q}_i,X_i) \in \mathcal{J}^{2,+}_{\mathrm{B}_{1}}u_i(x_i),
\]
\[
|x_i - \bar{x}_i|  \leq r, \quad \text{and} \quad |u_i(x_i)| + |\vec{q}_i| + \|\mathrm{X}_i\| \leq \mathrm{C}_{\star}.
\]
Then for each $\varepsilon > 0$, there exist $\mathrm{X}_i \in \mathrm{Sym}(n)$ such that:
\begin{itemize}
\item[(i)] $(D_{x_i}\varphi(\bar{x}_1,\ldots,\bar{x}_k),\mathrm{X}_i) \in \overline{\mathcal{J}}^{2,+}_{\mathrm{B}_{1}}u_i(\bar{x}_i)$ for $i=1,\dots,k$,
\item[(ii)] $- \left( \frac{1}{\varepsilon} + \|\mathrm{A}\| \right) \mathrm{Id}_n \leq \begin{pmatrix} \mathrm{X}_1 & & 0 \\ & \ddots & \\ 0 & & \mathrm{X}_k \end{pmatrix} \leq \mathrm{A} + \varepsilon \mathrm{A}^2$,
\end{itemize}
where $\mathrm{A} = D^2 \varphi(\bar{x}_1,\ldots,\bar{x}_k)$.
\end{lemma}

\subsection*{Existence of viscosity solutions}

Here, we focus on the existence and uniqueness of solutions to problems \eqref{Problem} and \eqref{Eq-Henon-type}. Beyond this goal, we introduce a more general Hamiltonian term than the one originally considered. Let \( \mathscr{H}: \mathbb{R}^n \times \Omega  \to \mathbb{R} \) be a continuous function satisfying the following properties:
\begin{itemize}
\item[(\textbf{A1})] \( |\mathscr{H}( \vec{\xi},x)| \leq \mathrm{C}(1+|\vec{\xi}|^{\gamma}+ |\vec{\xi}|^\kappa) \) for some \( 0<\gamma<\kappa \) and all \( (\vec{\xi}, x) \in \mathbb{R}^n \times \Omega \);
\item[(\textbf{A2})] \( |\mathscr{H}(\vec{\xi},x) - \mathscr{H}( \vec{\xi},y)| \leq \omega(|x - y|)(1 + |\vec{\xi}|^{\gamma}+|\vec{\xi}|^\kappa) \), where \( \omega: [0, \infty) \to [0, \infty) \) is a modulus of continuity, i.e., an increasing function such that \( \omega(0) = 0 \).
\end{itemize}

\begin{example}
We now present examples of Hamiltonians that satisfy conditions \(({\bf A1})\)--\(({\bf A2})\):
\begin{itemize}
\item [1.] \(\mathscr{H}(\vec{\xi},x)=\langle \mathfrak{B}(x),\vec{\xi}\rangle|\vec{\xi}|^{\theta}+\varrho(x)|\vec{\xi}|^{\sigma}\), for \(0< \sigma<1+\theta\), where \(\mathfrak{B}\) and \(\varrho\) are uniformly continuous in \(\Omega\). In this case, \(\gamma=\sigma\), \(\kappa=1+\theta\), and the constant \(\mathrm{C}\) in \(({\bf A1})\)--\(({\bf A2})\) depends only on \(\|\varrho\|_{L^{\infty}(\Omega)}\) and \(\|\mathfrak{B}\|_{L^{\infty}(\Omega;\mathbb{R}^{n})}\). Moreover, the modulus of continuity \(\omega\) depends on that of \(\mathfrak{B}\) and \(\varrho\).

\item[2.]  \(\mathscr{H}(\vec{\xi},x)=\mathfrak{a}(x)|\vec{\xi}|^{\theta}+\mathfrak{b}(x)|\vec{\xi}|^{\sigma}\), for \(0< \theta<\sigma\), where \(\mathfrak{a}\) and \(\mathfrak{b}\) are uniformly continuous in \(\Omega\).

\item[3.]  \(\mathscr{H}(\vec{\xi},x)=\mathfrak{a}(x)\frac{(1-|\vec{\xi}|^{2})^{2}-1}{1+|\vec{\xi}|^{2}}+\mathfrak{b}(x)\langle\mathfrak{B}(x)\vec{\xi},\vec{\xi}\rangle^{\frac{\sigma}{2}}\), for \(0<\sigma<2\) or \(\sigma > 2\), where \(\mathfrak{a}\) and \(\mathfrak{b}\) are uniformly continuous functions, and \(\mathfrak{B}:\mathbb{R}^{n}\to\operatorname{Sym}(n)\) is uniformly continuous such that, for each \(x\in \Omega\), the matrix \(\mathfrak{B}(x)\) is nonnegative definite.
\end{itemize}
\end{example}

We therefore study viscosity solutions to equations of the form
\[
|\nabla v|^{\theta} \Delta_p^{\mathrm{N}} v + \mathscr{H}( \nabla v,x) + f_0(v,x) = 0 \quad \text{in} \quad \Omega, \quad \text{and} \quad v = g \quad \text{on} \quad \partial \Omega, \tag{2.1}
\]
where \( f_0 \) and \( g \) are assumed to be continuous.

The archetypal model we consider is
\[
\mathcal{Q}\, v \coloneqq |\nabla v|^{\theta}\left( \Delta_p^{\mathrm{N}} v + \langle\vec{\mathfrak{B}}(x) , \nabla v\rangle\right) + \varrho(x)|\nabla v|^{\sigma} = \mathfrak{a}(x) v_+^m(x) + \mathfrak{h}(x),
\]
where we assume conditions \((\textbf{H1})\)--\((\textbf{H4})\), with \(\mathfrak{h} \in \mathrm{C}^0(\Omega)\).

Inspired by \cite[Lemma 2.14]{ADaSS25} and \cite{BisVo23}, we introduce a fundamental tool that ensures the existence and uniqueness of viscosity solutions to problem \eqref{Problem} (resp. \eqref{Eq-Henon-type}), which will also be essential for establishing weak geometric properties in the following sections. This general formulation will be instrumental in our future investigations.

\begin{lemma}[\bf Comparison Principle]\label{Comp-Princ}
Let $\mathfrak{c}, \mathfrak{f}_{1}, \mathfrak{f}_{2} \in \mathrm{C}^0(\overline{\Omega})$, and let $\mathfrak{F} : \mathbb{R} \to \mathbb{R}$ be a continuous and increasing function such that $\mathfrak{F}(0) = 0$. Suppose that $\mathfrak{u}, \mathfrak{v} \in \mathrm{C}^{0,1}_{\mathrm{loc}}(\Omega)$ are functions satisfying
\begin{equation}\label{CP}
\left\{
\begin{array}{rcll}
|\nabla \mathfrak{u}|^{\theta} \Delta_p^{\mathrm{N}} \mathfrak{u} + \mathscr{H}(\nabla \mathfrak{u}, x) + \mathfrak{c}(x) \mathfrak{F}(\mathfrak{u}) & \geq & \mathfrak{f}_1(x) & \text{in } \Omega, \\[0.2cm]
|\nabla \mathfrak{v}|^{\theta} \Delta_p^{\mathrm{N}} \mathfrak{v} + \mathscr{H}(\nabla \mathfrak{v}, x) + \mathfrak{c}(x) \mathfrak{F}(\mathfrak{v}) & \leq & \mathfrak{f}_2(x) & \text{in } \Omega,
\end{array}
\right.
\end{equation}
in the viscosity sense, where $\mathscr{H}$ is a Hamiltonian function satisfying assumptions \textnormal{(\textbf{A1})} and \textnormal{(\textbf{A2})}. Furthermore, assume that $\mathfrak{v} \geq \mathfrak{u}$ on $\partial \Omega$, and that one of the following conditions holds:
\begin{enumerate}
\item[\textnormal{\textbf{(a)}}] $\mathfrak{c} < 0$ in $\overline{\Omega}$ and $\mathfrak{f}_1 \geq \mathfrak{f}_2$ in $\overline{\Omega}$;
\item[\textnormal{\textbf{(b)}}] $\mathfrak{c} \leq 0$ in $\overline{\Omega}$ and $\mathfrak{f}_1 > \mathfrak{f}_2$ in $\overline{\Omega}$.
\end{enumerate}
Then, $\mathfrak{u} \geq \mathfrak{v}$ in $\Omega$.
\end{lemma}

\begin{proof}
We will proceed with the proof using a \textit{reductio ad absurdum} argument. Specifically, assume that there exists a constant $\mathrm{M}_0 > 0$ such that
\begin{equation}
\mathrm{M}_0 \coloneqq \sup_{\overline{\Omega}} (\mathfrak{u} - \mathfrak{v}) > 0.\label{condideM0}
\end{equation}
For each $\varepsilon > 0$, define
\begin{equation}\label{maximum}
\mathrm{M}_\varepsilon = \sup_{(x,y) \in \overline{\Omega} \times \overline{\Omega}} \left( \mathfrak{u}(x) - \mathfrak{v}(y) - \frac{1}{2\varepsilon} |x - y|^2 \right) < \infty.
\end{equation}
Consider $(x_\varepsilon, y_\varepsilon) \in \overline{\Omega} \times \overline{\Omega}$ as the point where the supremum \(\mathrm{M}_{\varepsilon}\) is attained. Following the same approach as in \cite[Lemma 3.1]{CIL}, we know that
\begin{equation}\label{CP1}
\lim_{\varepsilon \to 0} \frac{1}{\varepsilon} |x_\varepsilon - y_\varepsilon|^2 = 0 \quad \text{and} \quad \lim_{\varepsilon \to 0} \mathrm{M}_\varepsilon = \mathrm{M}_0.
\end{equation}
In particular,
\begin{equation}\label{CP2}
x_0 \coloneqq \lim_{\varepsilon \to 0} x_\varepsilon = \lim_{\varepsilon \to 0} y_\varepsilon,
\end{equation}
where $\mathfrak{u}(x_0) - \mathfrak{v}(x_0) = \mathrm{M}_0$. From \eqref{condideM0}, it follows that
\[
\sup_{\partial \Omega} (\mathfrak{u} - \mathfrak{v}) \leq 0 < \mathrm{M}_0.
\]

Furthermore, since the maximizer cannot approach the boundary, there exists a compact subset $\Omega_1 \Subset \Omega$ such that $x_\varepsilon, y_\varepsilon \in \Omega_1$ for all sufficiently small $\varepsilon > 0$. By the Ishii-Lions Lemma \ref{IshiiLions}, there exist matrices $\mathrm{X}, \mathrm{Y} \in \text{Sym}(n)$ such that
\begin{equation}\label{CP3}
\left( \frac{x_\varepsilon - y_\varepsilon}{\varepsilon}, \mathrm{X} \right) \in \overline{\mathcal{J}}^{2,+}_{\Omega_1} \mathfrak{u}(x_\varepsilon) \quad \text{and} \quad \left( \frac{y_\varepsilon - x_\varepsilon}{\varepsilon}, \mathrm{Y} \right) \in \overline{\mathcal{J}}^{2,-}_{\Omega_1} \mathfrak{v}(y_\varepsilon),
\end{equation}
and
\begin{equation}\label{CP4}
-\frac{3}{\varepsilon} \begin{pmatrix}
\mathrm{Id}_n & 0 \\
0 & \mathrm{Id}_n
\end{pmatrix} \leq \begin{pmatrix}
\mathrm{X} & 0 \\
0 & -\mathrm{Y}
\end{pmatrix} \leq \frac{3}{\varepsilon} \begin{pmatrix}
\mathrm{Id}_n & -\mathrm{Id}_n \\
-\mathrm{Id}_n & \mathrm{Id}_n
\end{pmatrix}.
\end{equation}
In particular, we obtain $\mathrm{X} \leq \mathrm{Y}$.

Now, given that $\mathfrak{u}$ and $\mathfrak{v}$ are Lipschitz continuous in $\Omega_1$, there exists a positive constant $\mathfrak{L} \defeq \max\{[\mathfrak{u}]_{\mathrm{\mathrm{C}^{0, 1}(\Omega_1)}}, [\mathfrak{v}]_{\mathrm{\mathrm{C}^{0, 1}(\Omega_1)}}\}$ such that
\[
|u(z_1) - u(z_2)| + |v(z_1) - v(z_2)| \leq \mathfrak{L}|z_1 - z_2| \quad \text{for all} \quad z_1, z_2 \in \Omega_1.
\]

Moreover, using the inequality
\[
\mathfrak{u}(x_\varepsilon) - \mathfrak{v}(x_\varepsilon) \leq \mathfrak{u}(x_\varepsilon) - \mathfrak{v}(y_\varepsilon) - \frac{1}{2\varepsilon}|x_\varepsilon - y_\varepsilon|^2,
\]
we deduce the following
\begin{equation}\label{Eq3.9}
\frac{1}{\varepsilon}|x_\varepsilon - y_\varepsilon| \leq 2\mathfrak{L}.
\end{equation}

Now, letting $\eta_\varepsilon = \frac{x_\varepsilon - y_\varepsilon}{\varepsilon}$ and using \eqref{CP} and \eqref{CP3}, we derive
\begin{align*}
\mathfrak{f}_1(x_\varepsilon) &\leq |\eta_\varepsilon|^\gamma \left( \operatorname{tr}(\mathrm{X}) + (p - 2) \left\langle \mathrm{X} \eta_\varepsilon, \eta_\varepsilon \right\rangle \right) + \mathfrak{c}(x_\varepsilon) \mathfrak{F}(\mathfrak{u}(x_\varepsilon)) + \mathscr{H}(\eta_\varepsilon,x_\varepsilon) \\
&\leq |\eta_\varepsilon|^\gamma \left( \operatorname{tr}(\mathrm{Y}) + (p - 2) \left\langle \mathrm{Y} \eta_\varepsilon, \eta_\varepsilon \right\rangle \right) + \mathfrak{c}(x_\varepsilon) \mathfrak{F}(\mathfrak{v}(x_\varepsilon)) + \mathscr{H}(\eta_\varepsilon,x_\varepsilon) \\
&\leq \mathfrak{f}_2(y_\varepsilon) - \mathfrak{c}(y_\varepsilon) \mathfrak{F}(\mathfrak{v}(y_\varepsilon)) - \mathscr{H}(\eta_\varepsilon,y_\varepsilon) + \mathfrak{c}(x_\varepsilon) \mathfrak{F}(\mathfrak{u}(x_\varepsilon)) + \mathscr{H}(\eta_\varepsilon,x_\varepsilon) \\
&\leq \mathfrak{f}_2(y_\varepsilon) + \mathfrak{F}(\mathfrak{v}(y_\varepsilon)) (\mathfrak{c}(x_\varepsilon) - \mathfrak{c}(y_\varepsilon)) + \left[ \min_{\overline{\Omega}} \mathfrak{c} \right] (\mathfrak{F}(\mathfrak{u}(x_\varepsilon)) - \mathfrak{F}(\mathfrak{v}(y_\varepsilon))) \\
&\quad + \omega(|x_{\varepsilon}-y_{\varepsilon}|)(1+|\eta_{\varepsilon}|^{\gamma}+|\eta_\varepsilon|^{\kappa}).
\end{align*}

Finally, by letting \(\varepsilon \to 0\) in the last inequality, and using the continuity of  \(\mathfrak{c}\), \(\mathfrak{B}\), and \(\varrho\), and recalling the convergences in  \eqref{CP1},  \eqref{CP2}, and invoking \eqref{Eq3.9}, we obtain
\[
\mathfrak{f}_1(x_0) - \mathfrak{f}_2(x_0) \leq \left[ \min_{\overline{\Omega}} \mathfrak{c} \right] (\mathfrak{F}(\mathfrak{u}(x_0)) - \mathfrak{F}(\mathfrak{v}(x_0))),
\]
which contradicts the assumptions $\textbf{(a)}$ and $\textbf{(b)}$. Therefore, we conclude that $\mathfrak{u} \geq \mathfrak{v}$ in $\Omega$.
\end{proof}

\bigskip

\begin{remark} 
Although the assumption $\mathfrak{u}, \mathfrak{v} \in \mathrm{C}^{0, 1}_{\text{loc}}(\Omega)$ may seem restrictive, it is entirely consistent with the applications of the Comparison Principle that we will present in this article. As an example, a subsolution (resp. supersolution) to the problem \eqref{Problem} is the one that satisfies the Dirichlet problem
$$
\left\{
\begin{array}{rclcl}
|\nabla u^{\ast}|^{\theta} \Delta_p^{\mathrm{N}} u^{\ast} + \mathscr{H}(\nabla u^{\ast}, x) & = & -\|f\| & \text{in } & \Omega, \\
u^{\ast}(x) & = & g(x) & \text{on } & \partial \Omega,
\end{array}
\right.
$$
and
$$
\left\{
\begin{array}{rclcl}
|\nabla u_{\ast}|^{\theta} \Delta_p^{\mathrm{N}} u_{\ast} + \mathscr{H}(\nabla u_{\ast}, x) & = & \|f\| & \text{in } & \Omega, \\
u_{\ast}(x) & = & g(x) & \text{on } & \partial \Omega.
\end{array}
\right.
$$
Hence, from Theorem \ref{Thm1.1}, such a sub/supersolution belongs to $\mathrm{C}_{\text{loc}}^{1, \alpha}(\Omega)$, i.e., they are locally Lipschitz continuous. In other words, we can always verify such a hypothesis to apply the Comparison Principle.
\end{remark}
\medskip

A natural consequence of the Comparison Principle is the existence of a viscosity solution to the Dirichlet problem associated with equation \eqref{Problem} (resp. \eqref{Eq-Henon-type}). This result is obtained via Perron's method, provided that a version of the Comparison Principle holds. 

The existence of such solutions can be established using standard techniques. Moreover, it is evident that $u^{\ast}$ and $u_{\ast}$ act as a supersolution and a subsolution, respectively, to \eqref{Problem}. Therefore, by invoking the Comparison Principle (Lemma \ref{Comp-Princ}), Perron's method guarantees the existence of a viscosity solution in $\mathrm{C}^0(\Omega)$ to \eqref{Problem} (resp. \eqref{Eq-Henon-type}). 

\section{Compactness and Approximation devices}\label{Section4}

In this section, we study the properties of solutions to the translated problem
\begin{equation}\label{problematransladado}
|\nabla w+\vec{q}|^{\theta}\Delta_{p}^{\mathrm{N}}(w+\langle\vec{q}, x\rangle)+\mathscr{H}(\nabla w+\vec{q},x)=\bar{f}(x)\quad\text{in } \mathrm{B}_{1},
\end{equation}
where \(\vec{q}\) is a vector in \(\mathbb{R}^{n}\).

We first establish the local Hölder continuity of solutions to problem \eqref{problematransladado}, under suitable control conditions on the Hamiltonian \(\mathscr{H}\) and the vector \(\vec{q}\).

\begin{lemma}[\bf Local Hölder Estimates]\label{Holderest}
Let \(\vec{q} \in \mathbb{R}^{n}\), and let \(w\) be a bounded viscosity solution to \eqref{problematransladado} in \(\mathrm{B}_{1}\). For any \(\mu \in (0,1)\), there exist positive constants \(\eta_{0}\) and \(\mathrm{C}_{\star}\), depending only on \(n\), \(p\), \(\mu\), \(\sigma\), and \(\theta\), such that if
\[
\|\mathfrak{B}\|_{L^{\infty}(\mathrm{B}_{1};\mathbb{R}^{n})} \leq \mathrm{C}_{\star} \quad \text{and} \quad \max\left\{\|\mathfrak{B}\|_{L^{\infty}(\mathrm{B}_{1};\mathbb{R}^{n})}(1+|\vec{q}|),\|\varrho\|_{L^{\infty}(\mathrm{B}_{1})}(1+|\vec{q}|^{\sigma-\theta})\right\} \leq \eta_{0},
\]
then \(w \in \mathrm{C}^{0,\mu}(\mathrm{B}_{15/16})\), and for all \(x, y \in \mathrm{B}_{15/16}\), the following estimate holds:
\begin{eqnarray*}
|w(x)-w(y)| \leq \mathrm{C}\left(\|w\|_{L^{\infty}(\mathrm{B}_{1})}+\|\bar{f}\|_{L^{\infty}(\mathrm{B}_{1})}^{\frac{1}{1+\theta}}+\|\varrho\|_{L^{\infty}(\mathrm{B}_{1})}^{\frac{1}{1+\theta-\sigma}}\right)|x-y|^{\mu},
\end{eqnarray*}
where \(\mu \in (0,1)\) depends only on \(n\), \(p\), \(\theta\), and \(\sigma\), and the constant \(\mathrm{C} > 0\) also depends on \(\mu\).
\end{lemma}

\begin{proof}
The main idea of the proof is to employ the viscosity method introduced by Ishii and Lions (cf. \cite{IL}) adopted to the current scenario of linear/sublinear Hamiltonian terms. 

We fix \(x_{0}, y_{0} \in \mathrm{B}_{\frac{15}{16}}\). For suitable positive constants \(\mathfrak{L}_{1}\) and \(\mathfrak{L}_{2}\), we define the function
\begin{eqnarray*}
\phi(x,y) := w(x) - w(y) - \mathfrak{L}_{2}\omega(|x - y|) - \frac{\mathfrak{L}_{1}}{2}\left(|x - x_{0}|^{2} + |y - y_{0}|^{2}\right),
\end{eqnarray*}
where the function \(\omega\) is given by \(\omega(t) = t^{\mu}\). We claim that \(\phi \leq 0\) in \(\overline{\mathrm{B}_{\frac{15}{16}}} \times \overline{\mathrm{B}_{\frac{15}{16}}}\). To prove this, we proceed by a \textit{reductio ad absurdum} argument. Suppose, by contradiction, that there exists a point \((\bar{x}, \bar{y})\) where \(\phi\) attains a positive maximum. In particular, \(\phi(\bar{x}, \bar{y}) > 0\) implies that \(\bar{x} \neq \bar{y}\).

Now, choosing
\begin{eqnarray*}
\mathfrak{L}_{1} \geq \frac{16\|w\|_{L^{\infty}(\mathrm{B}_{1})}}{\left(\min\left\{\dist(x_{0}, \partial \mathrm{B}_{15/16}), \dist(y_{0}, \partial \mathrm{B}_{15/16})\right\}\right)^{2}} =: \mathrm{C}\|w\|_{L^{\infty}(\mathrm{B}_{1})},
\end{eqnarray*}
we ensure that
\begin{eqnarray*}
|\bar{x} - x_{0}| \leq \frac{\dist(x_{0}, \partial \mathrm{B}_{15/16})}{2} \quad \text{and} \quad |\bar{y} - y_{0}| \leq \frac{\dist(y_{0}, \partial \mathrm{B}_{15/16})}{2}.
\end{eqnarray*}
Consequently, the points \(\bar{x}\) and \(\bar{y}\) belong to \(\mathrm{B}_{\frac{15}{16}}\). 

With these preliminary observations in place, we apply the Ishii–Lions Lemma \ref{IshiiLions} to the functions
\[
\tilde{u}(x) := w(x) - \frac{\mathfrak{L}_{1}}{2}|x - x_{0}|^{2} \quad \text{and} \quad \tilde{v}(y) := -w(y) - \frac{\mathfrak{L}_{1}}{2}|y - y_{0}|^{2},
\]
and obtain
\begin{eqnarray*}
\left(\varsigma_{1}, \mathrm{X} + \mathfrak{L}_{1} \mathrm{Id}_{n}\right) \in \overline{\mathcal{J}}^{2,+}_{\mathrm{B}_{\frac{15}{16}}}(w)(\bar{x}) \quad \text{and} \quad \left(\varsigma_{2}, \mathrm{Y} - \mathfrak{L}_{1} \mathrm{Id}_{n}\right) \in \overline{\mathcal{J}}^{2,-}_{\mathrm{B}_{\frac{15}{16}}}(w)(\bar{y}).
\end{eqnarray*}
Here,
\begin{eqnarray*}
\varsigma_{1} = \mathfrak{L}_{2} \omega'(|\bar{x} - \bar{y}|)\frac{\bar{x} - \bar{y}}{|\bar{x} - \bar{y}|} + \mathfrak{L}_{1}(\bar{x} - x_{0}), \\
\varsigma_{2} = \mathfrak{L}_{2} \omega'(|\bar{x} - \bar{y}|)\frac{\bar{x} - \bar{y}}{|\bar{x} - \bar{y}|} - \mathfrak{L}_{1}(\bar{y} - y_{0}).
\end{eqnarray*}

Imposing the condition $\mathfrak{L}_{2} > \frac{\mathfrak{L}_{1} \cdot 2^{4 - \mu}}{\mu}$, we obtain
\begin{eqnarray}\label{relacaodosvarsigma}
\frac{\mathfrak{L}_{2}}{2} \mu |\bar{x} - \bar{y}|^{\mu - 1} \leq |\varsigma_{i}| \leq 2 \mathfrak{L}_{2} \mu |\bar{x} - \bar{y}|^{\mu - 1}, \quad i = 1,2.
\end{eqnarray}

Moreover, by the Ishii–Lions Lemma \ref{IshiiLions}, there exist matrices $\mathrm{X}, \mathrm{Y} \in \mathrm{Sym}(n)$ such that for any $\tau > 0$ satisfying $\tau \mathrm{Z} < \mathrm{Id}_n$, it holds that
\begin{equation}\label{ineqmatrices}
-\frac{2}{\tau} \begin{pmatrix} \mathrm{Id}_n & 0 \\ 0 & \mathrm{Id}_n \end{pmatrix}
\leq 
\begin{pmatrix} \mathrm{X} & 0 \\ 0 & -\mathrm{Y} \end{pmatrix}
\leq 
\begin{pmatrix} \mathrm{Z}^{\tau} & -\mathrm{Z} \\ -\mathrm{Z} & \mathrm{Z}^{\tau} \end{pmatrix},
\end{equation}
where
\begin{align*}
\mathrm{Z} &= \mathfrak{L}_{2} \omega''(|\bar{x} - \bar{y}|) \frac{\bar{x} - \bar{y}}{|\bar{x} - \bar{y}|} \otimes \frac{\bar{x} - \bar{y}}{|\bar{x} - \bar{y}|} 
+ \mathfrak{L}_{2} \omega'(|\bar{x} - \bar{y}|) \left( \mathrm{Id}_n - \frac{\bar{x} - \bar{y}}{|\bar{x} - \bar{y}|} \otimes \frac{\bar{x} - \bar{y}}{|\bar{x} - \bar{y}|} \right) \\
&= \mathfrak{L}_{2} \mu |\bar{x} - \bar{y}|^{\mu - 2} \left( \mathrm{Id}_n + (\mu - 2) \frac{\bar{x} - \bar{y}}{|\bar{x} - \bar{y}|} \otimes \frac{\bar{x} - \bar{y}}{|\bar{x} - \bar{y}|} \right),
\end{align*}

and
\begin{equation*}
\mathrm{Z}^{\tau} = (\mathrm{Id}_n - \tau \mathrm{Z})^{-1} \mathrm{Z}.
\end{equation*}

Choosing $\kappa = \left(2\mathfrak{L}_{2} \mu |\bar{x} - \bar{y}|^{\mu - 2}\right)^{-1}$ and applying the Sherman–Morrison formula (see \cite{SherMorr50}), we find
\begin{eqnarray*}
\mathrm{Z}^{\tau} = 2 \mathfrak{L}_{2} \mu |\bar{x} - \bar{y}|^{\mu - 2} \left( \mathrm{Id}_n - 2 \frac{2 - \mu}{3 - \mu} \frac{\bar{x} - \bar{y}}{|\bar{x} - \bar{y}|} \otimes \frac{\bar{x} - \bar{y}}{|\bar{x} - \bar{y}|} \right).
\end{eqnarray*}

Setting $\xi = \frac{\bar{x} - \bar{y}}{|\bar{x} - \bar{y}|}$, it is straightforward to verify that
\begin{eqnarray*}
\langle \mathrm{Z}^{\kappa} \xi, \xi \rangle = 2 \mathfrak{L}_{2} \mu |\bar{x} - \bar{y}|^{\mu - 2} \cdot \frac{\mu - 1}{3 - \mu} < 0,
\end{eqnarray*}
since $\mu \in (0,1)$. Furthermore, from inequality \eqref{ineqmatrices}, it follows that
\begin{eqnarray}\label{estmatricesXY}
\mathrm{X} \leq \mathrm{Y} \quad \text{and} \quad \max\{\|\mathrm{X}\|, \|\mathrm{Y}\|\} \leq 4 \mathfrak{L}_{2} \mu |\bar{x} - \bar{y}|^{\mu - 2}.
\end{eqnarray}

Defining \(\zeta_{i} = \varsigma_{i} + \vec{q}\), from inequality \eqref{relacaodosvarsigma} we deduce
\begin{eqnarray}\label{relacaodosvarrho}
\frac{\mathfrak{L}_{2}}{4} \mu |\bar{x} - \bar{y}|^{\mu - 1} \leq |\zeta_{i}| \leq 3 \mathfrak{L}_{2} \mu |\bar{x} - \bar{y}|^{\mu - 1}, \quad i = 1,2,
\end{eqnarray}
provided that
\begin{eqnarray}\label{conddeL2}
\mathfrak{L}_{2} \geq \frac{4 \left[ \left( \eta_{0} \|\varrho\|_{L^{\infty}(B_{1})}^{-1} \right)^{\frac{1}{\sigma - \theta}} + \eta_{0} \|\mathfrak{B}\|_{L^{\infty}(B_{1}; \mathbb{R}^{n})}^{-1} \right]}{\mu}.
\end{eqnarray}

In this context, we use the following notation for each $\eta \in \mathbb{R}^{n}$:
\begin{eqnarray*}
\mathfrak{A}(\eta) = \mathrm{Id}_n + (p - 2)\frac{\eta}{|\eta|} \otimes \frac{\eta}{|\eta|}.
\end{eqnarray*}
Observe that the eigenvalues of the matrix \( \mathfrak{A}(\eta) \) lie in the open interval \( (\min\{1, p - 1\}, \max\{1, p - 1\}) \). 

Using the information from the sub-/super-differentials, we obtain:
\[
|\zeta_{1}|^{\theta} \operatorname{tr}(\mathfrak{A}(\zeta_1)(\mathrm{X} + \mathfrak{L}_{1} \mathrm{Id}_n)) + \mathscr{H}(\zeta_{1}, \bar{x}) \geq \bar{f}(\bar{x}),
\]
\[
|\zeta_{2}|^{\theta} \operatorname{tr}(\mathfrak{A}(\zeta_2)(\mathrm{Y} - \mathfrak{L}_{1} \mathrm{Id}_n)) + \mathscr{H}(\zeta_{2}, \bar{y}) \leq \bar{f}(\bar{y}).
\]

In other words, by the structure of the function \(\psi\), we can write:
\[
\operatorname{tr}(\mathfrak{A}(\zeta_1)(\mathrm{X} + \mathfrak{L}_{1} \mathrm{Id}_n)) + |\zeta_{1}|^{-\theta} \mathscr{H}(\zeta_{1}, \bar{x}) \geq |\zeta_{1}|^{-\theta} \bar{f}(\bar{x}),
\]
\[
\operatorname{tr}(\mathfrak{A}(\zeta_2)(\mathrm{Y} - \mathfrak{L}_{1} \mathrm{Id}_n)) + |\zeta_{2}|^{-\theta} \mathscr{H}(\zeta_{2}, \bar{y}) \leq |\zeta_{2}|^{-\theta} \bar{f}(\bar{y}).
\]

Combining these two inequalities, rearranging the terms, and using the definition of the Hamiltonian \(\mathscr{H}\), we conclude:
\begin{eqnarray}
0 &\leq& \|\bar{f}\|_{L^{\infty}(\mathrm{B}_{1})} (|\zeta_{1}|^{-\theta} + |\zeta_{2}|^{-\theta}) + \|\mathfrak{B}\|_{L^{\infty}(\mathrm{B}_{1}; \mathbb{R}^{n})} (|\zeta_{1}| + |\zeta_{2}|) \nonumber \\
&& + \|\varrho\|_{L^{\infty}(\mathrm{B}_{1})} (|\zeta_{1}|^{\sigma - \theta} + |\zeta_{2}|^{\sigma - \theta}) + \mathfrak{L}_{1}[\operatorname{tr}(\mathfrak{A}(\zeta_{1})) + \operatorname{tr}(\mathfrak{A}(\zeta_{2}))] \nonumber \\
&& + \operatorname{tr}(\mathfrak{A}(\zeta_{1})(\mathrm{X} - \mathrm{Y})) + \operatorname{tr}((\mathfrak{A}(\zeta_{1}) - \mathfrak{A}(\zeta_{2}))\mathrm{Y}) \nonumber \\
&=:& \mathfrak{I}_{1} + \mathfrak{I}_{2} + \mathfrak{I}_{3} + \mathfrak{I}_{4} + \mathfrak{I}_{5} + \mathfrak{I}_{6}. \label{estholder1}
\end{eqnarray}

In the sequel, we will estimate each term \( \mathfrak{I}_j \), for \( j = 1, \dots, 6 \):

\begin{itemize}
\item({\bf Estimate of \(\mathfrak{I}_{1}\)})\\
Due to the choice of the lower bound for \(|\zeta_{i}|\), we have that
\begin{eqnarray*}
|\zeta_{i}|^{-\theta}\leq \left(\eta_{0}\|\varrho\|_{L^{\infty}(\mathrm{B}_{1})}\right)^{-\frac{\theta}{\sigma-\theta}}|\bar{x}-\bar{y}|^{\theta(1-\mu)}\leq 2^{\theta(1-\mu)},
\end{eqnarray*}
since \(\left(\eta_{0}\|\varrho\|_{L^{\infty}(\mathrm{B}_{1})}^{-1}\right)^{\frac{1}{\sigma-\theta}}\geq 1\) and \eqref{conddeL2} holds. Thus, we can conclude that
\begin{eqnarray}
\mathfrak{I}_{1}\leq 2^{1+\theta(1-\mu)}\|\bar{f}\|_{L^{\infty}(\mathrm{B}_{1})}\label{estI1}.
\end{eqnarray}

\item({\bf Estimate of \(\mathfrak{I}_{2}\)})\\
We can use the estimate \eqref{relacaodosvarrho} to conclude that
\begin{eqnarray}
\mathfrak{I}_{2}&\leq&2\|\mathfrak{B}\|_{L^{\infty}(\mathrm{B}_{1};\mathbb{R}^{n})}( 3\mathfrak{L}_{2}\mu|\bar{x}-\bar{y}|^{\mu-1})\nonumber\\
&\leq&6\|\mathfrak{B}\|_{L^{\infty}(\mathrm{B}_{1};\mathbb{R}^{n})}\mathfrak{L}_{2}\mu|\bar{x}-\bar{y}|^{\mu-2}|\bar{x}-\bar{y}|\nonumber\\
&\leq&12\|\mathfrak{B}\|_{L^{\infty}(\mathrm{B}_{1};\mathbb{R}^{n})}\mathfrak{L}_{2}\mu|\bar{x}-\bar{y}|^{\mu-2}.\label{estI2}
\end{eqnarray}

\item({\bf Estimate of \(\mathfrak{I}_{3}\)})\\
From the hypotheses, we have that
\begin{eqnarray}
\mathfrak{I}_{3}\leq 2\|\varrho\|_{L^{\infty}(B_{1})}( 3\mathfrak{L}_{2}\mu|\bar{x}-\bar{y}|^{\mu-1})^{\sigma-\theta}\label{estI3}
\end{eqnarray}

\item({\bf Estimate of \(\mathfrak{I}_{4}\)})\\
Since the eigenvalues of \( \mathfrak{A}(\zeta_i) \) belong to the open interval \( (\min\{1, p-1\}, \max\{1, p-1\}) \), we have that
\begin{eqnarray}
\operatorname{tr}(\mathfrak{A}(\zeta_{i}))\in(n\min\{1, p-1\},n\max\{1, p-1\}),\,\, i=1,2.\label{estimateAesp}
\end{eqnarray}
In particular,
\begin{eqnarray}
\mathfrak{I}_{4}\leq 2\mathfrak{L}_{1}n\max\{1,p-1\}\label{estI4}.
\end{eqnarray}

\item ({\bf Estimate of \(\mathfrak{I}_{5}\)})\\
We use the notation \( \lambda_i(\mathrm{M}) \) for the \( i \)-th eigenvalue of a matrix \( \mathrm{M} \). The matrix inequalities \eqref{ineqmatrices} applied to the vector \( (\xi, -\xi) \) for \( \xi = \frac{\bar{x} - \bar{y}}{|\bar{x} - \bar{y}|} \), give
\[
\langle(\mathrm{X}-\mathrm{Y})\xi,\xi\rangle\leq 4\langle \mathrm{Z}^{\tau}\xi,\xi\rangle\leq-8\mathfrak{L}_{2}\mu|\bar{x}-\bar{y}|^{\mu-2}\left(\frac{1-\mu}{3-\mu}\right)<0.
\]
Consequently,
\[
\lambda_{i}(\mathrm{X}-\mathrm{Y})\leq -8\mathfrak{L}_{2}\mu|\bar{x}-\bar{y}|^{\mu-2}\left(\frac{1-\mu}{3-\mu}\right),\,\, i=1,\ldots,n. 
\]
Hence,
\begin{eqnarray}
\mathfrak{I}_{5}&\leq&\sum_{i=1}^{n}\lambda_{i}(\mathfrak{A}(\zeta_{1}))\lambda_{i}(\mathrm{X}-\mathrm{Y})\nonumber\\
&\leq&\min\{1,p-1\}\min_{1\leq i\leq n}\lambda_{i}(\mathrm{X}-\mathrm{Y})\nonumber\\
&\leq&-8\mathfrak{L}_{2}\mu\left(\frac{1-\mu}{3-\mu}\right)\min\{1,p-1\}|\bar{x}-\bar{y}|^{\mu-2}\label{estI5}.
\end{eqnarray}
\item({\bf Estimate of \(\mathfrak{I}_{6}\)})\\
First, we know from \cite[Lemma 3.1]{Attou18} that
\[
\|\mathfrak{A}(\zeta_{1})-\mathfrak{A}(\zeta_{2})\|\leq \frac{16|p-2|\mathfrak{L}_{1}}{\mathfrak{L}_{2}\mu|\bar{x}-\bar{y}|^{\mu-1}}.
\]
Thus,
\begin{eqnarray}
\mathfrak{I}_{6} &\leq& n\|\mathfrak{A}(\zeta_{1})-\mathfrak{A}(\zeta_{2})\|\|\mathrm{Y}\|\nonumber\\
&\leq& n\frac{|p-2|\mathfrak{L}_{1}}{\mathfrak{L}_{2}\mu|\bar{x}-\bar{y}|^{\mu-1}}4\mathfrak{L}_{2}\mu|\bar{x}-\bar{y}|^{\mu-2}\nonumber\\
&=& 64n|p-2|\mathfrak{L}_{1}|\bar{x}-\bar{y}|^{-1}.\label{estI6}
\end{eqnarray}
\end{itemize}

Hence, by substituting \eqref{estI1}, \eqref{estI2}, \eqref{estI3}, \eqref{estI4}, \eqref{estI5}, and \eqref{estI6} into \eqref{estholder1} and rearranging the terms, we obtain
\begin{eqnarray}
0 &\leq& 2^{1+\theta(1-\mu)}\|\bar{f}\|_{L^{\infty}(\mathrm{B}_{1})} + |\bar{x}-\bar{y}|^{\mu-2}\mathfrak{L}_{2}\mu \Bigg[ -8\left(\frac{1-\mu}{3-\mu}\right)\min\{1,p-1\}\nonumber\\
&+& 12\|\mathfrak{B}\|_{L^{\infty}(\mathrm{B}_{1};\mathbb{R}^{n})} + \frac{2\|\varrho\|_{L^{\infty}(\mathrm{B}_{1})}}{(\mathfrak{L}_{2}\mu|\bar{x}-\bar{y}|^{\mu-1})^{1-\sigma+\theta}}\nonumber \\
&+& \frac{2n}{\mathfrak{L}_{2}\mu|\bar{x}-\bar{y}|^{\mu-2}}\mathfrak{L}_{1} + \frac{64n|p-2|}{\mathfrak{L}_{2}\mu|\bar{x}-\bar{y}|^{\mu-1}}\mathfrak{L}_{1} \Bigg]\label{contradicao}
\end{eqnarray}

Now, we choose \(\mathrm{C}_{\star}=\left(\frac{1-\mu}{12(3-\mu)}\right)\min\{1,p-1\}\) and \(\mathfrak{L}_{2}\) large enough such that
$$
\begin{cases}
\mathfrak{L}_{2}\mu|\bar{x}-\bar{y}|^{\mu-1} \geq \left(\frac{2\|\varrho\|_{L^{\infty}(\mathrm{B}_{1})}}{\min\{1,p-1\}\left(\frac{1-\mu}{3-\mu}\right)}\right)^{\frac{1}{1-\sigma+\theta}}\\
\mathfrak{L}_{2}\mu|\bar{x}-\bar{y}|^{\frac{\mu-2}{1+\theta}} \geq \frac{\mathfrak{L}_{1}^{\frac{1}{1+\theta}} + 8^{\frac{1}{1+\theta}}\|\bar{f}\|_{L^{\infty}(\mathrm{B}_{1})}^{\theta}}{\left(\min\{1,p-1\}\left(\frac{1-\mu}{3-\mu}\right)\right)^{\frac{1}{1+\theta}}} \\
\mathfrak{L}_{2}\mu|\bar{x}-\bar{y}|^{\mu-2} \geq \frac{64n|p-2|+2n}{\min\{1,p-1\}\left(\frac{1-\mu}{3-\mu}\right)}\mathfrak{L}_{1}
\end{cases}
$$
With these choices, we obtain in \eqref{contradicao} that
\begin{eqnarray*}
0 &\leq& -3\left(\frac{1-\mu}{3-\mu}\right)\min\{1,p-1\}\mathfrak{L}_{2}\mu|\bar{x}-\bar{y}|^{\mu-2} < 0,
\end{eqnarray*}
which yields a contradiction. Moreover, by the above choices, it follows that
\[
\mathfrak{L}_{2} \geq \mathrm{C}(n,p,\mu, \theta,\sigma)\left(\|u\|_{L^{\infty}(\mathrm{B}_{1})} + \|\bar{f}\|_{L^{\infty}(\mathrm{B}_{1})}^{\frac{1}{1+\theta}} + \|\varrho\|_{L^{\infty}(\mathrm{B}_{1})}^{\frac{1}{1+\theta-\sigma}}\right).
\]
Finally, as \(\phi \leq 0\) in \(\overline{\mathrm{B}_{\frac{15}{16}}}\times \overline{\mathrm{B}_{\frac{15}{16}}}\), in particular,
\[
|w(x)-w(y)| \leq \mathfrak{L}_{2} \omega(|x-y|) \leq \mathrm{C}\left(\|w\|_{L^{\infty}(\mathrm{B}_{1})} + \|\varrho\|_{L^{\infty}(\mathrm{B}_{1})}^{\frac{1}{1+\theta-\sigma}} + \|\bar{f}\|_{L^{\infty}(\mathrm{B}_{1})}^{\frac{1}{1+\theta}}\right) |x-y|^{\mu}, 
\]
for all \(x, y \in \mathrm{B}_{\frac{15}{16}}\).
\end{proof}

\begin{remark}
Upon carefully re-examining the estimates provided in the previous lemma, it becomes apparent that the proposed strategy fails to accommodate both the sublinear and superlinear cases within a unified framework. Given the current inability to establish H\"{o}lder regularity estimates in the superlinear regime, we plan to revisit this issue in a forthcoming study by exploring an alternative approach.
\end{remark}

We can utilize the previous result as a compactness tool to derive an approximation result for solutions to problem \eqref{problematransladado} via well-regularized profiles, specifically by functions of class \(\mathrm{C}_{\text{loc}}^{1,\alpha_{\mathrm{Hom}}^{(p)}}\). In summary, the following approximation result can be stated:

\begin{lemma}[{\bf Approximation Lemma}]\label{approxlemma}
Let \(w\) be a normalized viscosity solution to \eqref{problematransladado} in \(\mathrm{B}_{1}\) for some \(\vec{q} \in \mathbb{R}^{n}\). Given any \(\varepsilon > 0\), there exists a universal constant \(\eta_{0} > 0\) such that, if
\[
\max\left\{\|f\|_{L^{\infty}(\mathrm{B}_{1})},\ \|\mathfrak{B}\|_{L^{\infty}(\mathrm{B}_{1};\mathbb{R}^{n})}(1 + |\vec{q}|),\ \|\varrho\|_{L^{\infty}(\mathrm{B}_{1})}(1 + |\vec{q}|^{\sigma - \theta})\right\} \leq \eta,
\]
then there exists a function \(\mathfrak{h} \in \mathrm{C}^{1, \alpha_{\mathrm{Hom}}^{(p)}}(\mathrm{B}_{3/4})\) such that
\[
\sup_{\mathrm{B}_{\frac{1}{2}}} |w - \mathfrak{h}| \leq \varepsilon.
\]
\end{lemma}

\begin{proof}
Let us prove the result by \textit{Reductio ad Absurdum}. Suppose, for the sake of contradiction, that the statement fails. Then, there exist \(\varepsilon_{0} > 0\), sequences of functions \((w_{k})_{k \in \mathbb{N}}\), \((f_{k})_{k \in \mathbb{N}}\), \((\mathscr{H}_{k})_{k \in \mathbb{N}}\), and a sequence of vectors \((\vec{q}_{k})_{k \in \mathbb{N}}\) satisfying
\begin{equation}\label{soldeuk}
|\nabla w_{k} + \vec{q}_{k}|^{\theta} \Delta_{p}^{\mathrm{N}}(w_{k} + \langle \vec{q}_{k}, x \rangle) + \mathscr{H}_{k}(\nabla w_{k} + \vec{q}_{k}, x) = f_{k}(x) \quad \text{in } \mathrm{B}_{1},
\end{equation}
where:
\begin{itemize}
    \item[(i)] \(\mathscr{H}_{k}(x,\xi) = \langle \mathfrak{B}_{k}(x), \xi \rangle |\xi|^{\theta} + \varrho_{k}(x) |\xi|^{\sigma}\),
    \item[(ii)] \(\max\left\{\|f_{k}\|_{L^{\infty}(\mathrm{B}_{1})},\, \|\mathfrak{B}_{k}\|_{L^{\infty}(\mathrm{B}_{1};\mathbb{R}^{n})}(1 + |\vec{q}_{k}|),\, \|\varrho_{k}\|_{L^{\infty}(\mathrm{B}_{1})}(1 + |\vec{q}_{k}|^{\sigma - \theta})\right\} \leq \frac{1}{k}\),
    \item[(iii)] \(\|w_{k}\|_{L^{\infty}(\mathrm{B}_{1})} \leq 1\).
\end{itemize}

However,
\begin{equation}\label{contradition}
\sup_{\mathrm{B}_{1/2}} |w_{k} - \mathfrak{h}| > \varepsilon_{0}
\end{equation}
for all \(\mathfrak{h} \in \mathrm{C}^{1,\alpha_{\mathrm{Hom}}^{(p)}}(\mathrm{B}_{3/4})\).

From (ii), it follows that \(f_{k} \to 0\) in \(L^{\infty}\), and \(\|f_{k}\|_{L^{\infty}(\mathrm{B}_{1})} \leq 1\) for all \(k \in \mathbb{N}\). Moreover, by the Archimedean property, for sufficiently large \(k\), we have \(\frac{1}{k} \leq \min\{\eta_{0}, \mathrm{C}_{\star}\}\), where \(\eta_{0}\) and \(\mathrm{C}_{\star}\) are the constants appearing in the local Hölder regularity Lemma~\ref{Holderest}.

Therefore, by Lemma~\ref{Holderest}, the sequence \((w_{k})\) is precompact in the \(\mathrm{C}^{0}(\mathrm{B}_{7/8})\)-topology for sufficiently large \(k\). Hence, by the Arzelà–Ascoli Theorem (possibly after passing to a subsequence), we obtain that \(w_{k} \to w_{\infty}\) uniformly in \(\mathrm{B}_{7/8}\) for some limit function \(w_{\infty}\).

We now analyze the sequence \((\vec{q}_{k})_{k \in \mathbb{N}}\) in two distinct scenarios:

\begin{itemize}
\item \textbf{Case I:} \((\vec{q}_{k})_{k\in\mathbb{N}}\) is an unbounded sequence.\\
In this case, we claim that \(w_{\infty}\) is a viscosity solution to 
\[
\operatorname{tr}(\mathcal{A}D^{2}w_{\infty})=0 \quad \text{in } \mathrm{B}_{7/8},
\]
for some symmetric matrix \(\mathcal{A}\). To this end, consider a (relabeled) subsequence \((\vec{q}_{k})_{k\in\mathbb{N}}\) such that \(|\vec{q}_{k}|\geq k\) for all \(k\in\mathbb{N}\); in particular, \(\vec{q}_{k}\neq 0\). Define the normalized sequence \(\vec{e}_{k}=\frac{\vec{q}_{k}}{|\vec{q}_{k}|}\). Note that \(|\vec{e}_{k}|=1\) for all \(k\in\mathbb{N}\), and thus, up to a subsequence, \(\vec{e}_{k}\to \vec{e}_{\infty}\) with \(|\vec{e}_{\infty}|=1\). 

We claim that \(w_{\infty}\) solves
\begin{eqnarray}\label{soldeplaplacenorm}
\operatorname{tr}((\mathrm{Id}_{n}+(p-2)\vec{e}_{\infty}\otimes\vec{e}_{\infty}) D^{2}w_{\infty})=0 \quad \text{in } \mathrm{B}_{7/8}.
\end{eqnarray}
Indeed, we will prove only that \( w_{\infty} \) is a supersolution of \eqref{soldeplaplacenorm} in the viscosity sense, as the subsolution case follows analogously.

Let us fix \( x_{0} \in \mathrm{B}_{7/8} \) and consider the quadratic polynomial
\[
\mathfrak{P}(x)=w_{\infty}(x_{0})+\mathfrak{b}_{\infty}\cdot (x-x_{0})+\frac{1}{2}(x-x_0)^{t}\mathrm{M}_{\infty}(x-x_{0}).
\]
Note that \(\mathfrak{P}(x_{0})=w_{\infty}(x_{0})\). Suppose that \(\mathfrak{P}\) touches \(w_{\infty}\) from below at \(x_{0}\). We may assume without loss of generality that \( \mathfrak{b}_{\infty} \neq 0 \) (see \cite{KMP12}). We must verify that
\[
\operatorname{tr}((\mathrm{Id}_{n}+(p-2)\vec{e}_{\infty}\otimes\vec{e}_{\infty}) \mathrm{M}_{\infty})\leq 0.
\]

For fixed \(r \ll 1\), let \(x_{k} \in \mathrm{B}_{r}\) be such that
\begin{eqnarray}
\mathfrak{P}(x_{k}) - w_{k}(x_{k}) = \max\{\mathfrak{P}(x) - w_{k}(x);\, x \in \mathrm{B}_{r}\}.\label{seqxk}
\end{eqnarray}
Furthermore, we may choose \(\vec{q}_{k}\) such that
\[
\max\{1,|\mathfrak{b}_{\infty}|\} < |\vec{q}_{k}|, \quad \forall k \geq k_{0}.
\]

Since \(w_{k}\) is a viscosity solution to \eqref{soldeuk}, and using condition \(\mathrm{(ii)}\), we obtain
\begin{eqnarray}\label{convpontula1}
\left|\frac{\mathfrak{b}_{\infty}}{|\vec{q}_{k}|}+\vec{e}_{k}\right|^{\theta}\Delta_{p}^{\mathrm{N}}\mathfrak{P}(x_{k})+\frac{\mathscr{H}_{k}(\mathfrak{b}_{\infty}+\vec{q}_{k},x_{k})}{|\vec{q}_{k}|^{\theta}}\leq \frac{f_{k}(x_{k})}{|\vec{q}_{k}|^{\theta}} \leq \|f_{k}\|_{L^{\infty}(\mathrm{B}_{1})}.
\end{eqnarray}
Moreover, we estimate
\begin{eqnarray*}
\left|\frac{\mathscr{H}_{k}(\mathfrak{b}_{\infty}+\vec{q}_{k},x_{k})}{|\vec{q}_{k}|^{\theta}}\right| 
&\leq& \|\mathfrak{B}_{k}\|_{L^{\infty}(\mathrm{B}_{1};\mathbb{R}^{n})} \left|\frac{\mathfrak{b}_{\infty}}{|\vec{q}_{k}|}+\vec{e}_{k}\right|^{1+\theta}|\vec{q}_{k}| \\
&+& \|\varrho_{k}\|_{L^{\infty}(\mathrm{B}_{1})} \left|\frac{\mathfrak{b}_{\infty}}{|\vec{q}_{k}|}+\vec{e}_{k}\right|^{\sigma} |\vec{q}_{k}|^{\sigma - \theta} \\
&\leq& \frac{1}{k} \left( \left|\frac{\mathfrak{b}_{\infty}}{|\vec{q}_{k}|}+\vec{e}_{k}\right|^{1+\theta} + \left|\frac{\mathfrak{b}_{\infty}}{|\vec{q}_{k}|}+\vec{e}_{k}\right|^{\sigma} \right) \\
&\leq& \frac{1}{k} \left[2^{\theta} \left( \left( \frac{|\mathfrak{b}_{\infty}|}{|\vec{q}_{k}|} \right)^{1+\theta} + 1 \right) + 2^{\sigma - 1} \left( \left( \frac{|\mathfrak{b}_{\infty}|}{|\vec{q}_{k}|} \right)^{\sigma} + 1 \right) \right] \\
&\leq& \frac{2^{2+\theta}}{k}, \quad \forall k \geq k_{0},
\end{eqnarray*}
which implies
\[
\left|\frac{\mathscr{H}_{k}(\mathfrak{b}_{\infty}+\vec{q}_{k},x_{k})}{|\vec{q}_{k}|^{\theta}}\right| \to 0.
\]
Therefore, letting \(k \to \infty\) in \eqref{convpontula1}, we conclude that
\[
\operatorname{tr}(\mathrm{M}_{\infty}) + (p - 2) \left\langle \mathrm{M}_{\infty} \frac{\vec{e}_{\infty}}{|\vec{e}_{\infty}|}, \frac{\vec{e}_{\infty}}{|\vec{e}_{\infty}|} \right\rangle \leq 0,
\]
that is,
\[
\operatorname{tr} \left( (\mathrm{Id}_{n} + (p - 2)\vec{e}_{\infty} \otimes \vec{e}_{\infty}) \mathrm{M}_{\infty} \right) \leq 0,
\]
since \(|\vec{e}_{\infty}| = 1\). Hence, \(w_{\infty}\) is a viscosity solution to \eqref{soldeplaplacenorm}, with
\[
\mathcal{A} = \mathrm{Id}_{n} + (p - 2)\vec{e}_{\infty} \otimes \vec{e}_{\infty},
\]
which establishes the claim. Up to a rotation, we may assume that \(w_{\infty}\) is a harmonic function. In particular, \(w_{\infty} \in \mathrm{C}^{1,\alpha_{\mathrm{Hom}}^{(p)}}\), since it is actually of class \(\mathrm{C}^{2,\alpha}\) for some \(\alpha \in (0,1)\), by the classical Evans–Krylov Theorem (see Caffarelli–Cabré's book \cite{CC}). This contradicts \eqref{contradition} as \(k \to \infty\).
\item \textbf{Case II:} \((\vec{q}_{k})_{k\in\mathbb{N}}\) is a bounded sequence.\\
In this case, passing to a subsequence if necessary, we may assume that \(\vec{q}_{k}\to \vec{q}_{\infty}\). Moreover, the hypothesis \(\mathrm{(ii)}\) ensures that \(|\mathfrak{B}_{k}|\to 0\) uniformly in \(\mathrm{B}_{1}\), and \(\varrho_{k}\to 0\). Thus, by compactness and stability arguments (cf. \cite[Appendix]{Attou18}), we deduce that \(w_{\infty}\) satisfies
\[
|\nabla w_{\infty}+\vec{q}_{\infty}|^{\theta}\Delta_{p}^{\mathrm{N}}(w_{\infty}+\langle \vec{q}_{\infty},x\rangle)=0\,\,\, \mathrm{in}\,\,\, \mathrm{B}_{\frac{7}{8}}.
\]
Hence, by the Cutting Lemma (Lemma \ref{CT-Lemma}), we conclude that \(w_{\infty}\) solves
\[
\Delta_{p}^{\mathrm{N}}(w_{\infty}+\langle \vec{q}_{\infty},x\rangle)=0\,\,\, \mathrm{in}\,\,\, \mathrm{B}_{\frac{7}{8}},
\]
and therefore, by Lemma \ref{Holder-Cont-hom}, there exists \(\bar{\beta}\in (0,1)\) such that \(w_{\infty}\in \mathrm{C}^{1,\bar{\beta}}\). This again contradicts \eqref{contradition} as \(k\to \infty\).
\end{itemize}
This concludes the proof.
\end{proof}

\section{Gradient estimates: Proof of Theorem \ref{Thm1.1}}\label{Section5}

In this section, using the preliminary tools developed in Section \ref{Section4}, we present a proof of Theorem \ref{Thm1.1}.
\begin{proposition}\label{Prop4.1}
Let \(u \in \mathrm{C}^0(\mathrm{B}_1)\) be a normalized viscosity solution to problem \eqref{Problem}. Given a value \(\alpha\) such that
\[
\alpha \in \left(0, \alpha_{\mathrm{Hom}}^{(p)}\right) \cap \left(0, \frac{1}{1+\theta}\right],
\]
there exist universal constants \(\rho \in (0, 1/2]\) and \(\delta_0 > 0\) such that, if
\[
\max\left\{ \|f\|_{L^\infty(\mathrm{B}_1)}, \|\mathfrak{B}\|_{L^\infty(\mathrm{B}_1; \mathbb{R}^n)}, \|\varrho\|_{L^\infty(\mathrm{B}_1)} \right\} \leq \delta_0,
\]
then \(u\) is \(\mathrm{C}^{1,\alpha}\) at the origin. More precisely, there exists a positive constant \(C\) such that the affine function \(\ell(x) = \mathfrak{a} + \langle \mathfrak{b}, x \rangle\) satisfies
\[
\sup_{\mathrm{B}_r} |u - \ell| \leq C r^{1 + \alpha}, \quad \forall r \in (0, \rho].
\]
\end{proposition}

\begin{proof}
First, we consider the universal constant \(\mathrm{C}_p>0\) from \(\mathrm{C}^{1,\alpha_{\mathrm{Hom}}^{(p)}}\)-regularity of the class \(\mathbb{H}_{n}^{(p)}\). Next, we set \(\varepsilon > 0\), which will be chosen \textit{a posteriori}. We claim that there exists a sequence of affine functions \(\ell_k(x) = \mathfrak{a}_k + \langle \mathfrak{b}_k, x \rangle\) such that for all \(k \geq 0\):
\begin{itemize}
\item[(i)] \(\displaystyle \sup_{\mathrm{B}_{\rho_k}} |u - \ell_k| \leq \rho^{k(1+\alpha)}\),
\item[(ii)] \(|\mathfrak{a}_k - \mathfrak{a}_{k-1}| + \rho^{k-1} |\mathfrak{b}_k - \mathfrak{b}_{k-1}| \leq \mathrm{C}_p \rho^{(k-1)(1+\alpha)}\),
\end{itemize}
where 
\[
\rho = \min\left\{ \left( \frac{1}{2 \mathrm{C}_p} \right)^{\frac{1}{\alpha_{\mathrm{Hom}}^{(p)} - \alpha}}, \frac{1}{2} \right\}.
\]
Indeed, to this end, let us choose \(\varepsilon = \frac{1}{2} \rho^{1+\alpha}\), which determines the existence of a constant \(\eta > 0\) in the Approximation Lemma \ref{approxlemma}. From these constants, we can define the constant \(\delta_0\) as
\[
\delta_0 = \max\left\{ 1 + \left( \frac{\mathrm{C}_p}{1 - \rho^{\alpha}} \right), 1 + \left( \frac{\mathrm{C}_p}{1 - \rho^{\alpha}} \right)^{\sigma - \theta} \right\}^{-1} \eta.
\]
The proof will be carried out by induction. Defining \(\ell_{-1} = \ell_0 = 0\), we see that the base case \(k = 0\) is clearly satisfied, since \(u\) is a normalized solution. Now, assuming by the induction hypothesis that it holds for some \(k \in \mathbb{N}\), we define the following auxiliary function
\[
u_k(x) = \frac{(u - \ell_k)(\rho^k x)}{\rho^{k(1+\alpha)}}, \quad x \in \mathrm{B}_1.
\]
It is possible to verify that \(u_k\) is a normalized (by the induction hypothesis) viscosity solution to
\begin{eqnarray*}
|\nabla u_k + \vec{q}_k|^\theta \left[ \Delta u_k + (p - 2) \left\langle D^2 u_k \frac{\nabla u_k + \vec{q}_k}{|\nabla u_k + \vec{q}_k|}, \frac{\nabla u_k + \vec{q}_k}{|\nabla u_k + \vec{q}_k|} \right\rangle \right] \\
+ \mathscr{H}_k(\nabla u_k + \vec{q}_k, x) = f_k(x) \quad \text{in} \quad \mathrm{B}_1,
\end{eqnarray*}
where 
\begin{itemize}
\item \(\vec{q}_k = \frac{\vec{\mathfrak{b}}_k}{\rho^{k\alpha}}\),
\item \(\mathscr{H}_k(\xi, x) = \langle \mathfrak{B}_k(x), \xi \rangle |\xi|^\theta + \varrho_k(x) |\xi|^\sigma\),
\item \(\mathfrak{B}_k(x) = \rho^k \mathfrak{B}(\rho^k x)\) and \(\varrho_k(x) = \rho^{k(1 - \alpha(1 + \theta - \sigma))} \varrho(\rho^k x)\),
\item \(f_k(x) = \rho^{k(1 - \alpha(1 + \theta))} f(\rho^k x)\).
\end{itemize}

Let us verify that \( u_k \) satisfies the conditions of the Approximation Lemma \ref{approxlemma}. Observe that, by the choice of \(\delta_0\), it follows that
\begin{eqnarray}\label{estdeBk}
\|\mathfrak{B}_k\|_{L^\infty(\mathrm{B}_1; \mathbb{R}^n)}(1 + |\vec{q}_k|) &=& \|\mathfrak{B}\|_{L^\infty(\mathrm{B}_{\rho^k}; \mathbb{R}^n)} \rho^k \left( 1 + \rho^{-k\alpha} |\vec{b}_k| \right) \nonumber \\
&\leq& \|\mathfrak{B}\|_{L^\infty(\mathrm{B}_1; \mathbb{R}^n)} \rho^k \left( 1 + \rho^{-k\alpha} \sum_{i=1}^{k} |\vec{b}_j - \vec{b}_{j-1}| \right) \nonumber \\
&\leq& \|\mathfrak{B}\|_{L^\infty(\mathrm{B}_1; \mathbb{R}^n)} \left( \rho^k + \rho^{k(1-\alpha)} \frac{\mathrm{C}_p}{1 - \rho^\alpha} \right) \nonumber \\
&\leq& \|\mathfrak{B}\|_{L^\infty(\mathrm{B}_1; \mathbb{R}^n)} \left( 1 + \frac{\mathrm{C}_p}{1 - \rho^\alpha} \right) \nonumber \\
&\leq& \delta_0.
\end{eqnarray}
Similarly, we have
\begin{eqnarray}\label{estdevarrhok}
\|\varrho_k\|_{L^\infty(\mathrm{B}_1)}(1 + |\vec{q}_k|^{\sigma - \theta}) &\leq& \|\varrho\|_{L^\infty(\mathrm{B}_1)} \rho^{k(1 - \alpha(1 + \theta - \sigma))} \nonumber \\
&+& \|\varrho\|_{L^\infty(\mathrm{B}_1)} \rho^{k(1 - \alpha)} \left( \sum_{i=1}^{k} |\vec{b}_j - \vec{b}_{j-1}| \right)^{\sigma - \theta} \nonumber \\
&\leq& \|\varrho\|_{L^\infty(\mathrm{B}_1)} \left( 1 + \left( \frac{\mathrm{C}_p}{1 - \rho^\alpha} \right)^{\sigma - \theta} \right) \nonumber \\
&\leq& \eta.
\end{eqnarray}
Moreover, since \(\|f_k\|_{L^\infty(\mathrm{B}_1)} \leq \rho^{k(1 - \alpha(1 + \theta))} \|f\|_{L^\infty(\mathrm{B}_1)}\) (because \(\rho \in (0, 1)\) and \(\alpha \leq \frac{1}{1 + \theta}\)), we have, by these estimates, along with \eqref{estdeBk} and \eqref{estdevarrhok}, that
\begin{eqnarray*}
\max\left\{ \|\mathfrak{B}_k\|_{L^\infty(\mathrm{B}_1; \mathbb{R}^n)}(1 + |\vec{q}_k|), \|\varrho_k\|_{L^\infty(\mathrm{B}_1)}(1 + |\vec{q}_k|^{\sigma - \theta}), \|f_k\|_{L^\infty(\mathrm{B}_1)} \right\} \leq \eta.
\end{eqnarray*}
Hence, we can invoke the Approximation Lemma \ref{approxlemma} to conclude that there exists a function \(\mathfrak{h} \in \mathrm{C}^{1, \alpha_{\mathrm{Hom}}^{(p)}}(\mathrm{B}_{3/4})\) such that
\begin{eqnarray}\label{estdeukh}
\sup_{\mathrm{B}_{1/2}} |u_k - \mathfrak{h}| \leq \varepsilon.
\end{eqnarray}
In this case, we know that
\[
\|\mathfrak{h}\|_{\mathrm{C}^{1, \alpha_{\mathrm{H}}^{(p)}}(\mathrm{B}_{3/4})} \leq \mathrm{C}_p.
\]
In particular, defining \(\mathfrak{a}_* = \mathfrak{h}(0)\) and \(\vec{\mathfrak{b}}_* = \nabla \mathfrak{h}(0)\), we have that
\begin{eqnarray}\label{estdehl}
\sup_{\mathrm{B}_r} |\mathfrak{h} - \ell_*| \leq \mathrm{C}_p r^{1 + \alpha_{\mathrm{Hom}}^{(p)}}, \quad \forall r \in (0, 1/2],
\end{eqnarray}
where \(\ell_*(x) = \mathfrak{a}_* + \langle \vec{\mathfrak{b}}_*, x \rangle\), and
\begin{eqnarray}\label{estdoscoef}
|\mathfrak{a}_*| + |\vec{\mathfrak{b}}_*| \leq \mathrm{C}_p.
\end{eqnarray}

Now, define \(\mathfrak{a}_{k+1} = \mathfrak{a}_{k} + \rho^{k(1+\alpha)} \vec{\mathfrak{a}}_*\) and \(\vec{\mathfrak{b}}_{k+1} = \vec{\mathfrak{b}}_{k} + \rho^{k\alpha} \vec{\mathfrak{b}}_*\). This defines the affine function \(\ell_{k+1}\). By the estimate \eqref{estdoscoef} and the induction hypothesis, we have
\begin{eqnarray*}
|\mathfrak{a}_{k+1} - \mathfrak{a}_k| + \rho^k |\vec{\mathfrak{b}}_{k+1} - \vec{\mathfrak{b}}_k| &=& \rho^{k(1+\alpha)} |\mathfrak{a}_*| + \rho^k \rho^{k\alpha} |\vec{\mathfrak{b}}_*| \leq \rho^{k(1+\alpha)} \mathrm{C}_p.
\end{eqnarray*}
This proves that the affine function \(\ell_{k+1}\) satisfies condition \((ii)\). For condition \((i)\), we note the estimates \eqref{estdeukh}, \eqref{estdehl}, and the definition of the affine profile \(\ell_{k+1}\), which give
\begin{eqnarray*}
\sup_{\mathrm{B}_\rho} |u_k - \ell_*| &\leq& \sup_{\mathrm{B}_\rho} |u_k - \mathfrak{h}| + \sup_{\mathrm{B}_\rho} |h - \ell_*| \nonumber \\
&\leq& \varepsilon + \mathrm{C}_1 \rho^{1 + \alpha_{\mathrm{Hom}}^{(p)}} \nonumber \\
&\leq& \frac{1}{2} \rho^{1 + \alpha} + \left( \frac{1}{2 \rho^{\alpha_{\mathrm{Hom}}^{(p)} - \alpha}} \right) \rho^{1 + \alpha_{\mathrm{Hom}}^{(p)}} \nonumber \\
&=& \rho^{1 + \alpha}.
\end{eqnarray*}
Therefore, the last estimate implies
\begin{eqnarray*}
\sup_{\mathrm{B}_{\rho^{k+1}}} |u - \ell_{k+1}| = \rho^{k(1+\alpha)} \sup_{\mathrm{B}_\rho} |u_k - \ell_*| \leq \rho^{(k+1)(1+\alpha)},
\end{eqnarray*}
which proves the claim.

Now, by the existence of a sequence of affine functions satisfying conditions \((i)\) and \((ii)\), in particular, the second condition implies that the sequences \((\mathfrak{a}_k)_{k \in \mathbb{N}}\) and \((\vec{\mathfrak{b}}_k)_{k \in \mathbb{N}}\) are Cauchy sequences in \(\mathbb{R}\) and \(\mathbb{R}^n\), respectively. Hence, there exist
\[
\mathfrak{a}_\infty = \lim_{k \to \infty} \mathfrak{a}_k \quad \text{and} \quad \vec{\mathfrak{b}}_\infty = \lim_{k \to \infty} \vec{\mathfrak{b}}_k,
\]
with the following rates of convergence:
\begin{eqnarray}
|\mathfrak{a}_k - \mathfrak{a}_\infty| &\leq& \frac{\mathrm{C}_p}{1 - \rho^{1 + \alpha}} \rho^{k(1 + \alpha)}, \quad \text{and} \quad |\vec{\mathfrak{b}}_k - \vec{\mathfrak{b}}_\infty| \leq \frac{\mathrm{C}_p}{1 - \rho^\alpha} \rho^{k\alpha}. \label{orderofconvergence}
\end{eqnarray}

As a result, the affine function \(\ell(x) = \mathfrak{a}_\infty + \langle \vec{\mathfrak{b}}_\infty, x \rangle\) is well-defined.

Finally, given \(r \in (0, \rho]\), we can find \(k \in \mathbb{N}\) such that \(\rho^{k+1} < r \leq \rho^k\). Then, using condition \((i)\) and the estimate \eqref{orderofconvergence}, we conclude that
\begin{eqnarray*}
\sup_{\mathrm{B}_r} |u - \ell| &\leq& \sup_{\mathrm{B}_r} |u - \ell_k| + \sup_{\mathrm{B}_r} |\ell_k - \ell| \\
&\leq& \sup_{\mathrm{B}_{\rho^k}} |u - \ell_k| + \sup_{x \in \mathrm{B}_r} \left( |\mathfrak{a}_k - \mathfrak{a}_\infty| + |x| |\vec{\mathfrak{b}}_k - \vec{\mathfrak{b}}_\infty| \right) \\
&\leq& \rho^{k(1 + \alpha)} + \frac{\mathrm{C}_p}{1 - \rho^{1 + \alpha}} \rho^{k(1 + \alpha)} + r \frac{\mathrm{C}_p}{1 - \rho^\alpha} \rho^{k \alpha} \\
&\leq& \frac{1}{\rho^{1 + \alpha}} \left( 1 + \frac{\mathrm{C}_p}{1 - \rho^{1 + \alpha}} + \frac{\mathrm{C}_p}{1 - \rho^\alpha} \right) r^{1 + \alpha} \\
&\leq& \mathrm{C} r^{1 + \alpha},
\end{eqnarray*}
where
\[
\mathrm{C} = \frac{1}{\rho^{1 + \alpha}} \left( 1 + \frac{2 \mathrm{C}_p}{1 - \rho^{1 + \alpha}} \right).
\]
This completes the proof.
\end{proof}

Finally, we are in a position to address the proof of Theorem \ref{Thm1.1}.

\begin{proof}[{\bf Proof of Theorem \ref{Thm1.1}}]
Fix \( x_0 \in \mathrm{B}_{1/2} \) and define the function \( w: \mathrm{B}_1 \to \mathbb{R} \) by
\[
w(x) = \frac{u(x_{0}+rx)}{\tau},
\]
for parameters \(0 < r \leq \frac{1}{2}\) and \(\tau > 0\) to be determined \textit{a posteriori}. It is straightforward to verify that \(w\) satisfies
\[
|\nabla w|^{\theta} \Delta_{p}^{\mathrm{N}} w + \mathscr{H}_{\tau,r}(\nabla w, x) = f_{\tau,r}(x) \quad \text{in } \mathrm{B}_1,
\]
where
\begin{itemize}
    \item \(\mathscr{H}_{\tau,r}(\xi,x) = \langle \mathfrak{B}_{\tau,r}(x), \xi \rangle |\xi|^{\theta} + \varrho_{\tau,r}(x) |\xi|^{\sigma},\)
    \item \(\mathfrak{B}_{\tau,r}(x) = r \mathfrak{B}(x_0 + rx)\) and \(\varrho_{\tau,r}(x) = \dfrac{r^{2+\theta-\sigma}}{\tau^{1+\theta-\sigma}} \varrho(x_0 + rx),\)
    \item \(f_{\tau,r}(x) = \dfrac{r^{2+\theta}}{\tau^{1+\theta}} f(x_0 + rx).\)
\end{itemize}

Observe that
\begin{equation}\label{estdeftaur}
\|f_{\tau,r}\|_{L^{\infty}(\mathrm{B}_{1})} = \frac{r^{2+\theta}}{\tau^{1+\theta}} \|f\|_{L^{\infty}(\mathrm{B}_{r}(x_0))} \leq \frac{r^{2+\theta}}{\tau^{1+\theta}} \|f\|_{L^{\infty}(\mathrm{B}_{1})},
\end{equation}
since \(\mathrm{B}_r(x_0) \subset \mathrm{B}_1\). Regarding the vector field \(\mathfrak{B}_{\tau,r}\), note that
\[
\|\mathfrak{B}_{\tau,r}\|_{L^{\infty}(\mathrm{B}_1)} \leq \|\mathfrak{B}\|_{L^{\infty}(\mathrm{B}_1)},
\]
because \(0 < r \leq \frac{1}{2}\). Furthermore,
\begin{equation}\label{estdevarrhotaur}
\|\varrho_{\tau,r}\|_{L^{\infty}(\mathrm{B}_1)} \leq \frac{r^{2+\theta-\sigma}}{\tau^{1+\theta-\sigma}} \|\varrho\|_{L^{\infty}(\mathrm{B}_1)}.
\end{equation}

Now, let \(\delta_0 > 0\) be the constant from Proposition \ref{Prop4.1}, and define the constants \(\tau\) and \(r\) by
\[
\tau \coloneqq \|u\|_{L^{\infty}(\mathrm{B}_1)} + \|\varrho\|^{\frac{1}{1+\theta-\sigma}}_{L^{\infty}(\mathrm{B}_1)} + \|f\|^{\frac{1}{1+\theta}}_{L^{\infty}(\mathrm{B}_1)} \quad \text{and} \quad 
r \coloneqq \min \left\{ \frac{\delta_0}{\|\mathfrak{B}\|_{L^{\infty}(\mathrm{B}_1; \mathbb{R}^n)} + 1}, \, \delta_0^{\frac{1}{2+\theta}}, \, \delta_0^{\frac{1}{2+\theta-\sigma}}, \, \frac{1}{2} \right\}.
\]

With this choice of parameters, \(w\) falls under the hypotheses of Proposition \ref{Prop4.1}. Consequently, both \(w\) and \(u\) belong to \(\mathrm{C}^{1,\alpha}\) at \(x_0\). Moreover, by a standard covering argument, we deduce that \(u \in \mathrm{C}^{1,\alpha}(\mathrm{B}_{1/2})\), and it satisfies the following regularity estimate:
\[
[u]_{\mathrm{C}^{1,\alpha}(\mathrm{B}_{1/2})} \leq \mathrm{C} \left( \|u\|_{L^{\infty}(\mathrm{B}_1)} + \|\varrho\|^{\frac{1}{1+\theta-\sigma}}_{L^{\infty}(\mathrm{B}_1)} + \|f\|^{\frac{1}{1+\theta}}_{L^{\infty}(\mathrm{B}_1)} \right),
\]
where \(\mathrm{C} > 0\) is a universal constant.
\end{proof}

\section{Sharp regularity estimates: Proof of Theorem \ref{Opt_reg-Extremum-point}}

This section is devoted to addressing sharp regularity estimates along extremum points, namely. Theorem \ref{Opt_reg-Extremum-point}.

\begin{proof}[{\bf Proof of Theorem \ref{Opt_reg-Extremum-point}}] 
For simplicity, and without loss of generality, we assume \( x_0 = \vec{0} \). After performing the translation \( v(x) \coloneqq u(x) - u(x_0) \), where \( x_0 \) is a minimum point, it is well known that proving estimate \ref{Higher Reg'} is equivalent to verifying the existence of a constant \( \mathrm{C} > 0 \) such that, for all \( j \in \mathbb{N} \), the following geometric iteration holds (see, for instance, \cite{CKS00}):
\begin{equation}\label{higher1}
\mathfrak{s}_{j+1} \leq \max\left\{\mathrm{C} \cdot 2^{-\hat{\beta}(j+1)}, 2^{-\hat{\beta}} \mathfrak{s}_j\right\},
\end{equation}
where
\begin{equation*}\label{beta_hat}
\mathfrak{s}_j := \sup_{\mathrm{B}_{2^{-j}}} v(x), \quad \text{and} \quad \hat{\beta} := 1 + \frac{1}{1+\theta} = \frac{2+\theta}{1+\theta}.
\end{equation*}

Suppose, for the sake of contradiction, that \eqref{higher1} does not hold. That is, for each \( k \in \mathbb{N} \), there exists \( j_k \in \mathbb{N} \) such that
\begin{equation}\label{higher2}
\mathfrak{s}_{j_k + 1} > \max\left\{k \cdot 2^{-\hat{\beta}(j_k + 1)}, 2^{-\hat{\beta}} \mathfrak{s}_{j_k}\right\}.
\end{equation}

Now, for each \( k \in \mathbb{N} \), define the rescaled function \( v_k: \mathrm{B}_1 \rightarrow \mathbb{R} \) as
\[
v_k(x) \coloneqq \frac{v(2^{-j_k} x)}{\mathfrak{s}_{j_k + 1}}.
\]
Observe that
\begin{eqnarray}
0 \leq v_k(x) \leq 2^{\hat{\beta}}, \nonumber\\
v_k(0) = 0, \label{vk2}\\ 
\sup_{\overline{\mathrm{B}_{\frac{1}{2}}}} v_k(x) = 1. \label{vk3}
\end{eqnarray}

Moreover, in the viscosity sense, we have
\begin{equation}\label{vksolution}
|\nabla v_k|^\theta \Delta_p^\mathrm{N} v_k + \mathscr{H}_{k}(\nabla v_k,x) =  f_k(x),
\end{equation}
where 
\begin{itemize}
    \item \(\mathscr{H}_{k}(\xi,x) = \langle\mathfrak{B}_{k}(x), \xi \rangle |\xi|^{\theta} + \varrho_{k}(x) |\xi|^{\sigma}\),
    \item \(\mathfrak{B}_{k}(x) = 2^{-j_k} \mathfrak{B}(2^{-j_k} x)\), \quad \(\varrho_{k}(x) = \frac{2^{-j_k(2+\theta-\sigma)}}{\mathfrak{s}_{j_k+1}^{1+\theta-\sigma}} \varrho(2^{-j_k} x)\),
    \item \(f_k(x) = \frac{2^{-j_k(2+\theta)}}{\mathfrak{s}_{j_k+1}^{1+\theta}} f(2^{-j_k} x)\).
\end{itemize}

Now, by Lemma \ref{Holderest}, there exists a function \( v_\infty \in \mathrm{C}^{0, \mu} \) such that \( v_k \to v_\infty \) locally uniformly. Moreover, \( v_\infty \) satisfies
\begin{equation}\label{Vasco}
v_\infty(0) = 0 \quad \text{and} \quad \sup_{\overline{\mathrm{B}_{\frac{1}{2}}}} v_\infty(x) = 1,
\end{equation}
as a consequence of \eqref{higher2} and \eqref{vk2}.

Using again \eqref{higher2}, we obtain
\[
\|f_k\|_{L^{\infty}(\mathrm{B}_1)} \leq 2^{2+\theta} \left(\frac{1}{k}\right)^{1+\theta} \|f\|_{L^{\infty}(\mathrm{B}_1)} \to 0 \quad \text{as} \quad k \to \infty.
\]
Additionally,
\[
\|\mathfrak{B}_k\|_{L^{\infty}(\mathrm{B}_1;\mathbb{R}^{n})} \leq 2^{-j_k} \|\mathfrak{B}\|_{L^{\infty}(\mathrm{B}_1;\mathbb{R}^{n})} \to 0 \quad \text{as} \quad k \to \infty.
\]
Moreover, applying \eqref{higher2} once again, we find
\[
\|\varrho_k\|_{L^{\infty}(\mathrm{B}_1)} \leq \frac{1}{k^{\theta+1-\sigma} \cdot 2^{\frac{(\theta+2)(\theta+1-\sigma) - 3\sigma j_k}{\theta+1}}} \|\varrho\|_{L^{\infty}(\mathrm{B}_1)} \to 0 \quad \text{as} \quad k \to \infty.
\]

Therefore, by the stability result (see \cite[Appendix]{Attou18}) combined with \eqref{vksolution}, we conclude that
\[
|\nabla v_\infty|^\theta \Delta_p^\mathrm{N} v_\infty = 0 \quad \text{in} \quad \mathrm{B}_{8/9}.
\]

Finally, by applying Lemma \ref{CT-Lemma}, Theorem \ref{Thm-Equiv-Sol}, and Vázquez's Strong Maximum Principle from \cite{Vazq84}, we deduce that \( v_\infty \equiv 0 \) in \( \mathrm{B}_{8/9} \), which contradicts \eqref{Vasco}. Hence, \eqref{higher1} must hold.

To conclude the proof, let \( r \in (0, 1) \), and choose \( j \in \mathbb{N} \) such that
\[
2^{-(j + 1)} \leq r \leq 2^{-j}.
\]
Then, using \eqref{higher1}, we establish the existence of a universal constant \( \mathrm{C} > 0 \) such that
\[
\sup_{\mathrm{B}_r} \left(u(x) - u(x_0)\right) = \sup_{\mathrm{B}_r} v \leq \sup_{\mathrm{B}_{2^{-j}}} v \leq \frac{\mathrm{C}}{2^{\hat{\beta}}} r^{\hat{\beta}}.
\]
This concludes the proof.
\end{proof}

\section{Non-degeneracy estimate: Proof of Theorem \ref{Thm-Non-Deg}}

In what follows, we establish the non-degeneracy property of viscosity solutions to \eqref{Problem}, namely Theorem \ref{Thm-Non-Deg}. As a relevant observation, this property holds even in the superlinear regime, i.e., $\sigma > 1+\theta$.  

\begin{proof}[{\bf Proof of Theorem \ref{Thm-Non-Deg}}]
Let \( x_0 \in B_1 \) be a local extremum point. Without loss of generality, we may assume that \( x_0 = 0 \) and that it is a local minimum. 

Now, define the following barrier function:
\[
\psi(x) = \kappa |x|^{\beta}, \quad \text{for}\quad x \in B_1,
\]
where \(\beta = \frac{2+\theta}{1+\theta}\), and \(\kappa > 0\) is a constant to be chosen later. It is straightforward to verify that:
\begin{itemize}
\item[\checkmark] \(\nabla \psi(x) = \kappa \beta |x|^{\beta-2} x\)
\item[\checkmark] \(D^2 \psi(x) = \kappa \beta |x|^{\beta-2} \left[ \mathrm{Id}_n - (2 - \beta) \frac{x}{|x|} \otimes \frac{x}{|x|} \right]\)
\end{itemize}

Hence, \(\psi\) satisfies (pointwise):
\begin{itemize}
\item In the sublinear case \(\theta < \sigma < 1+\theta\),
\begin{eqnarray}
|\nabla \psi|^{\theta} \Delta_{p}^{\mathrm{N}} \psi + \mathscr{H}(\nabla \psi, x) &\leq& \kappa^{\theta} \beta^{1+\theta} \Big[n - 1 + (\beta - 1)(p - 1) \nonumber \\
&+& \|\mathfrak{B}\|_{L^{\infty}(\mathrm{B}_1; \mathbb{R}^n)} + \beta^{\sigma - (1+\theta)} \|\varrho\|_{L^{\infty}(\mathrm{B}_1)} \Big] \nonumber \\
&<& \mathfrak{c}_0 \label{supersolucaopsi1}
\end{eqnarray}
by choosing \(\kappa\) such that
\[
\kappa < \min\left\{1, \frac{\mathfrak{c}_0^{1/\theta}}{\beta^{(1+\theta)/\theta} \left[n - 1 + (\beta - 1)(p - 1) + \|\mathfrak{B}\|_{L^{\infty}(\mathrm{B}_1; \mathbb{R}^n)} + \beta^{\sigma - (1+\theta)} \|\varrho\|_{L^{\infty}(\mathrm{B}_1)} \right]^{1/\theta}} \right\}.
\]

\item In the superlinear case \((1+\theta < \sigma )\),
\begin{eqnarray}
|\nabla \psi|^{\theta} \Delta_{p}^{\mathrm{N}} \psi + \mathscr{H}(\nabla \psi, x) &\leq& \kappa^{1+\theta} \beta^{1+\theta} |x|^{(\beta-1)(1+\theta)-1} \Big[n - 1 + (\beta - 1)(p - 1) \nonumber \\
&+& \|\mathfrak{B}\|_{L^{\infty}(\mathrm{B}_1; \mathbb{R}^n)} + (\kappa \beta)^{\sigma - (1+\theta)} \|\varrho\|_{L^{\infty}(\mathrm{B}_1)} \Big] \nonumber \\
&<& \mathfrak{c}_0 \label{supersolucaopsi2}
\end{eqnarray}
whenever we choose
\[
\kappa < \min\left\{1, \frac{\mathfrak{c}_0^{1/(1+\theta)}}{\beta \left[n - 1 + (\beta - 1)(p - 1) + \|\mathfrak{B}\|_{L^{\infty}(\mathrm{B}_1; \mathbb{R}^n)} + \beta^{\sigma - (1+\theta)} \|\varrho\|_{L^{\infty}(\mathrm{B}_1)} \right]^{1/(1+\theta)}} \right\}.
\]
\end{itemize}

Since the right-hand side of \eqref{Problem} does not depend on \( u \), we may, by a translation argument, assume without loss of generality that \( u(0) > 0 \). Indeed, for any \(\varepsilon > 0\), define \(v_{\varepsilon}(x) = u(x) - u(0) + \varepsilon\), so that \(v_{\varepsilon}(0) > 0\). Moreover, in the viscosity sense, we have
\[
|\nabla \psi|^{\theta} \Delta_{p}^{\mathrm{N}} \psi + \mathscr{H}(\nabla \psi, x) < \mathfrak{c} \leq f(x) = |\nabla v_{\varepsilon}|^{\theta} \Delta_{p}^{\mathrm{N}} v_{\varepsilon} + \mathscr{H}(\nabla v_{\varepsilon}, x).
\]

We claim that for any open ball \(\mathrm{B}_{r} \subset \mathrm{B}_{1}\), there exists a point \(y_{r} \in \partial \mathrm{B}_{r}\) such that \(v_{\epsilon}(y_{r}) \geq \psi(y_{r})\). Suppose, for contradiction, that this is not the case. Then we would have \(v_{\epsilon} < \psi\) on \(\partial \mathrm{B}_{r}\). 

Since \(\psi\) satisfies either \eqref{supersolucaopsi1} (sublinear case) or \eqref{supersolucaopsi2} (superlinear case), it follows that \(\psi\) is a strict supersolution of \eqref{Problem}. Moreover, note that \(v_{\epsilon}\) is also a (viscosity) subsolution of \eqref{Problem}. Thus, by the Comparison Principle (Theorem \ref{Comp-Princ}), we conclude that \(v_{\epsilon} \leq \psi\) in \(\mathrm{B}_{r}\). In particular,
\[
0 = \psi(0) \geq v_{\epsilon}(0) = \epsilon > 0,
\]
which is a contradiction. Therefore, there must exist \(y_{r} \in \partial \mathrm{B}_{r}\) such that \(v_{\epsilon}(y_{r}) \geq \psi(y_{r})\). 

This yields
\begin{eqnarray*}
\sup_{\partial \mathrm{B}_{r}} (u - u(0)) + \epsilon = \sup_{\partial \mathrm{B}_{r}} v_{\epsilon} \geq v_{\epsilon}(y_{r}) \geq \psi(y_{r}) = \kappa |y_{r}|^{\beta} = \kappa r^{\beta}.
\end{eqnarray*}
Finally, by letting \(\epsilon \to 0^{+}\), we obtain the desired result. The case where \(x_0\) is a local maximum holds by an analogous argument. This completes the proof.
\end{proof}

\section{Strong Maximum Principle: Proof of Theorem \ref{thm:SMP}}

This section is dedicated to establishing the Strong Maximum Principle/Hopf-type result for viscosity solutions of \eqref{Problem}, namely, Theorem \ref{thm:SMP}. Notably, this result holds even in the superlinear regime; that is, when $\sigma > 1+\theta$.  

\begin{proof}[{\bf Proof of Theorem \ref{thm:SMP}}]
Assume that \( v \) is not identically zero in \( \Omega \). We aim to show that \( v > 0 \) in \( \Omega \). Suppose, for the sake of contradiction, that there exist \(x_0 \in \Omega\) and \(r > 0\) such that \(\mathrm{B}_{2r}(x_0) \subset\subset \Omega \), \( v > 0 \) in \(\mathrm{B}_r(x_0) \), and \( v(z) = 0 \) for some \( z \in \partial \mathrm{B}_r(x_0) \). Without loss of generality, we assume \( x_0 = 0 \).

Define the function \( u(x) = e^{-\alpha |x|^2} - e^{-\alpha r^2} \), with \(\alpha > 0\) to be chosen \textit{a posteriori}. Clearly, \( u > 0 \) in \( \mathrm{B}_r \) and \( u = 0 \) on \(\partial \mathrm{B}_{r}\). Consider the differential operator
\[
\mathcal{L}[u](x)=|\nabla u(x)|^\theta\,\Delta_{p}^{\mathrm{N}}u(x) + \mathscr{H}(\nabla u(x),x)  + \mathfrak{c}(x)\,u^{1+\theta}(x).
\]
A straightforward computation shows that, for \( \frac{r}{2} \leq |x| \leq r \), the following inequalities hold:
\begin{itemize}
    \item If \(1+\theta<\sigma\), then
    \begin{eqnarray}
    \mathcal{L}[u](x) &\geq& (2\alpha e^{-\alpha|x|^2}|x|)^{1+\theta} \Bigg[ \alpha(p-1)r - \frac{2(n+p-2)}{r} \nonumber\\
    && - \|\mathfrak{B} \|_{L^{\infty}(\Omega;\mathbb{R}^{n})} - \|\varrho\|_{L^{\infty}(\Omega)}(2\alpha e^{-\alpha|x|^{2}}r)^{\sigma-1-\theta} \nonumber\\
    && - \| \mathfrak{c} \|_{L^{\infty}(\Omega)} \frac{1}{r}\left(1 - e^{-\alpha (r^2-|x|^2)} \right)^{1+\theta} \Bigg] > 0,\label{coddebarreirae1}
    \end{eqnarray}
    for suitably large \( \alpha \gg 1 \).
    
    \item If \(\theta<\sigma<1+\theta\), and assuming \(\varrho\geq 0\), then
    \begin{eqnarray}
    \mathcal{L}[u](x) &\geq& (2\alpha e^{-\alpha|x|^2}|x|)^{1+\theta} \Bigg[ \alpha(p-1)r - \frac{2(n+p-2)}{r} \nonumber\\
    && - \|\mathfrak{B} \|_{L^{\infty}(\Omega;\mathbb{R}^{n})} - \| \mathfrak{c} \|_{L^{\infty}(\Omega)} \frac{1}{r}\left(1 - e^{-\alpha (r^2-|x|^2)} \right)^{1+\theta} \Bigg] > 0. \label{coddebarreirae2}
    \end{eqnarray}
\end{itemize}

Observe that one can choose \( \mathrm{K}_0 \ll 1 \) such that \( \mathrm{K}_0 u \leq v \) on \( \partial \mathrm{B}_{r/2} \), since \( u > 0 \) on this boundary. Hence, by the Comparison Principle (Lemma \ref{Comp-Princ}), it follows that \(v \geq \mathrm{K}_0 u\) in the annular region \(\overline{\mathrm{B}}_{r} \setminus \mathrm{B}_{r/2}\).

Moreover, \( v \leq 0 \leq u \) in \( \mathrm{B}_{2r} \setminus \mathrm{B}_r \), so \( \mathrm{K}_0 u \) touches \( v \) from below at the point \( z \). Thus, invoking the definition of viscosity solutions (see Definition \ref{DefViscSol}) with the test function \(\mathrm{K}_0 u\), we obtain
\[
|\nabla(\mathrm{K}_0u(z))|^{\theta}\Delta_{p}^{\mathrm{N}} (\mathrm{K}_0 u(z)) + \mathscr{H}(\nabla (\mathrm{K}_0u(z)),z) + \mathfrak{c}(z)\, (\mathrm{K}_0 u(z))^{1+\theta} \leq 0,
\]
which contradicts the inequalities \eqref{coddebarreirae1} and \eqref{coddebarreirae2}. Hence, \( v > 0 \) in \( \Omega \).

To conclude the proof, we reiterate the above argument and use the fact that for any \( s > 0 \), there exists a constant \( \mathrm{C}_s > 0 \) such that \( 1 - e^{-t} \geq \mathrm{C}_s t \) for all \( t \in [0,s] \), thereby completing the proof.
\end{proof}

\section{H\'{e}non-type problems with strong absorption}

In this final section, we present the proof of Theorems \ref{Higher_continuity} and \ref{NDHTE} concerning qualitative properties of solutions to \eqref{Eq-Henon-type}. The first one is an improved regularity estimate along free boundary points.

\begin{proof}[{\bf Proof of Theorem \ref{Higher_continuity}}]
The proof of the result is similar to that of Theorem \ref{Opt_reg-Extremum-point}, so we will only comment on the significant changes in the proof. We may suppose without loss of generality that \(x_{0}=0\). To prove the estimate, it is sufficient to show that there exists a constant \(\mathrm{C}>0\) such that, for all \(j\in\mathbb{N}\)
\begin{equation}\label{higher1'}
\mathrm{s}_{j+1} \leq \max\left\{\mathrm{C} \cdot 2^{-\bar{\beta}(j+1)}, 2^{-\bar{\beta}} \mathrm{s}_j\right\},
\end{equation}
where
\begin{eqnarray*}
\mathrm{s}_{j}=\sup_{\mathrm{B}_{2^{-j}}}u\,\,\,\mbox{and}\,\,\, \bar{\beta}=\frac{2+\theta+\alpha}{1+\theta-m}.
\end{eqnarray*}
Suppose, for the sake of contradiction, that \eqref{higher1'} fails to hold, \textit{i.e.,} for each \( k \in \mathbb{N} \), there exists \( j_k \in \mathbb{N} \) such that
\begin{equation}\label{higher2'}
\mathrm{s}_{j_k + 1} > \max\left\{k 2^{-\bar{\beta}(j_k + 1)}, 2^{-\bar{\beta}} \mathrm{s}_{j_k}\right\}.
\end{equation}
Now, define the rescaled function \( u_k: \mathrm{B}_1 \rightarrow \mathbb{R} \) as
\[
u_k(x) \coloneqq \frac{u(2^{-j_k} x)}{\mathrm{s}_{j_k + 1}}. 
\]
As in the proof of Theorem \ref{Opt_reg-Extremum-point} we observe that for all \(k\in\mathbb{N}\), 
\begin{equation}\label{Bound-u_k}
0 \leq u_k(x) \leq 2^{\bar{\beta}} \quad \text{in} \quad B_{1} \quad \text{and} \quad  u_{k}(0)=0
\end{equation}
and \(\displaystyle\sup_{\overline{\mathrm{B}_{\frac{1}{2}}}} u_k(x) = 1\). Hereafter, \(u_{k}\) solves in the viscosity sense
\begin{align}
|\nabla u_k|^\theta \Delta_p^\mathrm{N} u_k +\mathscr{H}_{k}(\nabla u_{k},x)&= \left(\frac{2^{-j_k(2 + \theta)}}{\mathrm{s}_{j_k + 1}^{1 + \theta - m}}\right) \mathfrak{f}\left(2^{-j_k} x\right) (u_k)_+^m(x) \nonumber \\[0.2cm]
&\lesssim \left(\frac{2^{-j_k(2 + \theta + \alpha)}}{\mathrm{s}_{j_k + 1}^{1 + \theta - m}}\right) \mathrm{dist}(x, \mathrm{F})^\alpha (u_k)_+^m(x) \nonumber \\
&\coloneqq \Psi_k(x) \label{vksolution'},
\end{align}
where \(\mathscr{H}(\xi,x)=\langle \mathfrak{B}_{k}(x),\xi\rangle|\xi|^{\theta}+\varrho_{k}|\xi|^{\sigma}\) for
\[
\mathfrak{B}_{k}(x)=2^{-j_{k}}\mathfrak{B}(2^{-j_{k}}x)\,\,\,\mbox{and}\,\,\, \varrho_{k}(x)=\frac{2^{-j_{k}(2+\theta-\sigma)}}{\mathrm{s}_{j_{k}+1}^{1+\theta-\sigma}}\varrho(2^{-j_k}x).
\]
Note that the condition \eqref{higher2'} implies that
\[
\|\varrho_k\|_{L^{\infty}(\mathrm{B}_{1})}\leq \frac{2^{\bar{\beta}(1+\theta-\sigma)}}{k^{1+\theta-\sigma}2^{j_{k}(2+\theta-\sigma -(1+\theta-\sigma)\bar{\beta})}}\|\varrho\|_{L^{\infty}(\mathrm{B}_{1})}.
\]

In the conditions on \(\alpha\) and \(m\) implies that \(2+\theta-\sigma -(1+\theta-\sigma)\bar{\beta}\geq0\), it holds
\(\|\varrho_{k}\|_{L^{\infty}(\mathrm{B}_{1})}\to0\). 
Furthermore, it holds 
$$
\|\mathfrak{B}_k\|_{L^{\infty}(\mathrm{B}_1;\mathbb{R}^{n})} \leq 2^{-j_k}\|\mathfrak{B}\|_{L^{\infty}(\mathrm{B}_1;\mathbb{R}^{n})} \to 0 \quad \text{as} \quad k \to \infty.
$$
 By the compactness (Lemma \ref{Holderest}), we can find \( u_\infty \in \mathrm{C}_{\text{loc}}^{0, \mu} \) such that \( u_k \to u_{\infty} \) locally uniformly. Moreover, \( u_\infty \) satisfies
\begin{equation}\label{Vasco'}
u_\infty(0) = 0 \quad \text{and} \quad \sup_{\overline{\mathrm{B}_{\frac{1}{2}}}} u_\infty(x) = 1.
\end{equation}
Now, using \eqref{higher2'} and \eqref{Bound-u_k}, we obtain 
\[
\Psi_k \leq 2^{\bar{\beta}} \mathrm{diam}(\Omega) \frac{2^{2 + \theta + \alpha + \bar{\beta} m}}{k} \to 0 \quad \text{as} \quad k \to \infty.
\]
By stability (see \cite[Appendix]{Attou18}), together with \eqref{vksolution'}, we conclude that 
\[
|\nabla u_\infty|^\theta \Delta_p^\mathrm{N} u_\infty = 0 \quad \text{in} \quad \mathrm{B}_{8/9}.
\]
Finally, as in the proof of Theorem \ref{Opt_reg-Extremum-point}, it follows that \( u_{\infty} \equiv 0 \) in \( \mathrm{B}_{8/9} \), which contradicts \eqref{Vasco'}. Therefore, \eqref{higher1'} holds, which finishes the proof of the theorem.
\end{proof}

\begin{remark}\label{Example6.7}
Another relevant model to consider in \eqref{Eq-Henon-type} involves general distance-type weights and “noise terms,” and is given by the following problem:
\[
|\nabla u(x)|^{\theta} \Delta_p^\mathrm{N} u(x) + \mathscr{H}(\nabla u(x), x) = \sum_{i=1}^{k_0} \mathrm{dist}^{\alpha_i}(x, \mathrm{F}_i)\, u_{+}^{m_i}(x) + g_i(|x|) \quad \text{in} \quad \mathrm{B}_1,
\]
where \( \mathrm{F}_i \Subset \mathrm{B}_1 \) are closed sets, \( m_i \in [0, \theta + 1) \), \( \alpha_i > 0 \), and \( g_i: \mathrm{B}_1 \to \mathbb{R} \) are continuous functions such that
\[
\lim_{\substack{x \to \mathrm{F}_i \\ x \notin \mathrm{F}_i}} \frac{g_i(|x|)}{\mathrm{dist}(x, \mathrm{F}_i)^{\gamma_i}} = \mathfrak{L}_i \in [0, \infty) \quad \text{for some} \quad \gamma_i \geq 0.
\]

In this context, non-negative viscosity solutions belong to \( \mathrm{C}_{\mathrm{loc}}^{\kappa_0} \) along the set \( \mathfrak{A}_0 \), where
\[
\kappa_0 \coloneqq \min_{1 \leq i \leq k_0} \left\{ \frac{2 + \theta + \alpha_i}{1 + \theta - m_i}, \frac{2 + \theta + \gamma_i}{1 + \theta} \right\}, \quad \text{and} \quad \mathfrak{A}_0 \coloneqq \left( \bigcap_{i=1}^{k_0} \mathrm{F}_i \right) \cap \partial \{ u > 0 \} \cap \mathrm{B}_{1/2}.
\]
\end{remark}

\medskip

Finally, we address the non-degeneracy estimate for solutions of H\'{e}non-type models with strong absorption in the superlinear regime, i.e., $\sigma > 1+\theta$.  

\begin{proof}[{\bf Proof of Theorem \ref{NDHTE}}]
Let \( x_0 \in \overline{\{u > 0\}} \cap \mathrm{B}_{1/2} \). Without loss of generality, we may assume that \( x_0 = 0 \). Since \( u \) is continuous, we can suppose that \( x_0 \in \{u > 0\} \cap \mathrm{B}_{1/2} \). Define the barrier function
\[
\psi(x) = \kappa |x|^{\hat{\beta}},
\]
where \(\hat{\beta} = \frac{2 + \theta + \alpha}{1 + \theta - m}\), and the constant \(\kappa > 0\) is chosen such that:
\[
\kappa < \min\left\{1, \frac{\mathfrak{c}_{0}^{\frac{1}{1 + \theta - m}}}{\hat{\beta}^{\frac{1 + \theta}{1 + \theta - m}} \left[ n - 1 + (\hat{\beta} - 1)(p - 1) + \|\mathfrak{B}\|_{L^{\infty}(\mathrm{B}_{1}; \mathbb{R}^{n})} + \hat{\beta}^{\sigma - (1 + \theta)} \|\varrho\|_{L^{\infty}(\mathrm{B}_{1})} \right]^{\frac{1}{1 + \theta - m}}} \right\}.
\]

Hence, in the viscosity sense, we have
\[
|\nabla \psi|^{\theta} \Delta_{p}^{\mathrm{N}} \psi + \mathscr{H}(\nabla \psi, x) - \mathfrak{c}_{0} |x|^{\alpha} \psi_{+}^{m}(x) < 0 \leq |\nabla u|^{\theta} \Delta_{p}^{\mathrm{N}} u + \mathscr{H}(\nabla u, x) - \mathfrak{c}_{0} |x|^{\alpha} u_{+}^{m}(x).
\]

As in the proof of Theorem \ref{Thm-Non-Deg}, it is possible to verify that for any open ball \( \mathrm{B}_{r} \subset \mathrm{B}_{1/2} \), there exists \( y_r \in \partial \mathrm{B}_{r} \) such that \( u(y_r) \geq \psi(y_r) \). Finally, this implies that
\[
\sup_{\partial \mathrm{B}_{r}} u \geq u(y_r) \geq \psi(y_r) = \kappa |y_r|^{\hat{\beta}} = \kappa r^{\hat{\beta}},
\]
thereby completing the proof.
\end{proof}

\subsection*{Acknowledgments}

J. da Silva Bessa has been supported by FAPESP-Brazil under Grant No. 2023/18447-3. J.V. da Silva has received partial support from CNPq-Brazil under Grant No. 307131/2022-0, FAEPEX-UNICAMP (Project No. 2441/23, Special Calls - PIND - Individual Projects, 03/2023), and Chamada CNPq/MCTI No. 10/2023 - Faixa B - Consolidated Research Groups under Grant No. 420014/2023-3.






\end{document}